\documentclass[12pt,reqno]{article}

\usepackage[usenames]{color}
\usepackage{amssymb}
\usepackage{graphicx}
\usepackage{amscd}
\usepackage{url}
\usepackage{enumerate}
\usepackage{epstopdf}
\usepackage[colorlinks=true,
linkcolor=webgreen,
filecolor=webbrown,
citecolor=webgreen]{hyperref}

\definecolor{webgreen}{rgb}{0,.5,0}
\definecolor{webbrown}{rgb}{.6,0,0}

\usepackage{color}
\usepackage{fullpage}
\usepackage{float}

\usepackage{graphics,amsmath,amssymb}
\usepackage{amsthm}
\usepackage{amsfonts}
\usepackage{latexsym}

\usepackage{multirow}
\usepackage{algorithm}

\usepackage[noend]{algpseudocode}

\usepackage{booktabs} 
\usepackage{array} 

\newcommand{\seqnum}[1]{\href{http://oeis.org/#1}{\underline{#1}}}

\newcommand{\ra}[1]{\overrightarrow{#1}}
\newcommand{\la}[1]{\overleftarrow{#1}}

\begin{document}

\theoremstyle{plain}
\newtheorem{theorem}{Theorem}
\newtheorem{corollary}[theorem]{Corollary}
\newtheorem{lemma}[theorem]{Lemma}
\newtheorem{proposition}[theorem]{Proposition}
\newtheorem{conjecture}[theorem]{Conjecture}
\newtheorem{experiment}[theorem]{Experiment}
\newtheorem{remark}[theorem]{Remark}

\title{Extrema Property of the $k$-Ranking of Directed Paths and Cycles}
\author{Breeanne Baker Swart
\thanks{ Corresponding address:\ Department of Mathematics and Computer Science, The Citadel,
                Charleston, SC, 29409, Email:\ breeanne.swart@citadel.edu}\\
                The Citadel\\
                Charleston, SC, 29409
\and
Rigoberto Fl\'{o}rez
\thanks{The last and the first two authors were partially supported by Citadel Foundation.}
\thanks{ Corresponding address:\ Department of Mathematics and Computer Science, The Citadel,
                Charleston, SC, 29409, Email:\ rflorez1@citadel.edu}\\
                The Citadel\\
                Charleston, SC, 29409  \\
\and Darren A. Narayan
\thanks{Corresponding
address:\ School of Mathematical Sciences, Rochester Institute of
Technology, Rochester, NY 14623-5604,
Email:\ darren.narayan@rit.edu}\\
Rochester Institute of Technology\\
Rochester, NY 14623-5604\\
\and
George L. Rudolph
\thanks{ Corresponding address:\ Department of Mathematics and Computer Science, The Citadel,
                Charleston, SC, 29409, Email:\ george.rudolph@citadel.edu}\\
                The Citadel\\
                Charleston, SC, 29409  \\
} \maketitle

\begin{abstract}
A $k$-ranking of a directed graph $G$ is a labeling of the vertex set of $G$ with $k$ positive integers such
that every directed path connecting two vertices with the same label includes a vertex with a
larger label in between. The \textit{rank number of} $G$ is defined to be the smallest $k$ such that $G$
has a $k$-ranking. We find the largest possible directed graph that can be obtained from
a directed path or a directed cycle by attaching new edges to the vertices
such that the new graphs have the same rank number as the original graphs.  The adjacency matrix of the resulting graph is embedded in the Sierpi\'nski triangle.

We present a connection between the number of edges that can
be added to paths and the Stirling numbers of the second kind.  These results are generalized to create directed
graphs which are unions of directed paths and directed cycles that maintain the rank number of a base graph of a directed path or a directed cycle.
\end{abstract}

\section{Introduction}

A vertex coloring of a directed graph is a labeling of its vertices so that no two adjacent vertices
receive the same label. In a directed path, edges are oriented in the same direction. A $k$-ranking of a directed graph is a
labeling of the vertex set with $k$ positive integers such that for every directed path connecting two vertices with the same
label there is a vertex with a larger label in between.  A ranking is \textit{minimal} if the reduction of any label violates the ranking
property. The \textit{rank number} $\chi_{k}(G)$ of a directed graph $G$ is the smallest $k$ such that $G$ has a minimal
$k$-ranking.

It is known that the rank number of a graph $G$ may increase just by adding a new edge, even if the new edge and $G$ share vertices.
This raises the question ``what is the maximum size of a directed graph that satisfies the property that its rank number is equal to
the rank number of its largest directed subpath?'' Fl\'orez and Narayan \cite{floreznarayan, FN}
found results related to this question; however, the problem is still open. We believe that studying particular cases will
lead to a better understanding of the problem and potential solutions.

In this paper we study the necessary and sufficient conditions for the largest possible directed graph that can be obtained
by attaching edges to either a directed path or directed cycle without changing the rank number and the number of vertices.
In \cite{floreznarayan} there is a solution for the undirected case. Here, we analyze cases for which the new directed graph
keeps the rank number of the original graph. The maximum number of edges in such graphs is described as well as which edges are present in the graphs.

We build  families of directed graphs  by adding directed edges (called \textit{admissible}) to a directed path (called the \textit{base}). Those families
satisfy the condition that the graphs are maximal graphs with the property that the rank number of each graph equals the rank number of the base directed path.
The graphs of the first two families described were constructed recursively by adding all admissible edges  (without increasing the number of vertices) to a base directed path.
The same idea is extended to  directed cycles.

We generalize the concept developed to build the first four families above to other maximal families of graphs preserving the rank number of the base directed path by
adding admissible directed paths and directed cycles to a base directed path or directed cycle.

The number of edges and the number of admissible edges of the graphs in the first four families are counted using known numerical sequences. We prove, using the
 recursive construction, that the maximum number of edges in some of those families of graphs is given by a Stirling
number of the second kind.

For those who are interested in computational matters, we provide algorithms, some of which are given in terms of adjacency matrices. We found an interesting connection between one of the adjacency matrices and the Sierpi\'nski triangle. The adjacency matrix of the first graph found in this paper embeds naturally in the Sierpi\'nski sieve triangle.

\section{Preliminary Concepts} \label{sec:preliminaries}

In this section we review some known concepts and results, introduce new definitions, and give a proof for a lemma.

Let $V := \{v_{1},v_{2},\ldots,v_{\eta}\}$ be a set of vertices of a directed graph. An edge (arc) with vertices
$\{v_{i}, v_{j}\}\subseteq V$, $i < j$, with orientation $v_{i}\rightarrow v_{j}$ is denoted by $\overrightarrow{e}$ or by $\overrightarrow{v_{i}v_{j}}$,
and the edge with orientation $v_{i}\leftarrow v_{j}$ is denoted by $\overleftarrow{e}$ or by $\overleftarrow{v_{i}v_{j}}$. A directed path with vertices
$V$ is denoted by $\overrightarrow{P_{\eta}}$ if its edges are of the form $\overrightarrow{e}$.
A directed cycle with vertices $V$ is denoted by $\overrightarrow{C_{\eta}}$ if its edges are of the form $\overrightarrow{e}$.

Let $G$ be a finite directed graph with vertex set $V(G)$. A $k$-\emph{ranking} of $G$ is a labeling (or coloring) of $V(G)$
with $k$ positive integers such that every directed path that connects two vertices of the same label (color) contains a vertex
of a larger label (color). Thus, a labeling function  $f:V(G)\longrightarrow\{1,2,\ldots,k\}$ is a vertex $k$-\textit{ranking} of
$G$ if for all $u,v \in V(G)$ such that $u\not = v$ and $f(u)=f(v)$, then every  directed path connecting $u$ and $v$  contains a
vertex $w$ such that $f(w)>f(u).$ Like the chromatic number, the \textit{rank number} of a graph $G$ is
defined to be the smallest $k$ such that $G$ has a minimal $k$-ranking; it is denoted by $\chi_{k}(G)$.

Let $H_1$ and $H_2$ be directed graphs with $V(H_1)\subseteq V(H_2)$ and $E(H_1)\cap E(H_2)=\emptyset$.
We say that a directed edge $e \in E(H_1)$ is \emph{admissible} for $H_2$ if
$\chi_{r}\left(  H_2\cup\{e\}\right) = \chi_{r}(H_2)$, and $e$ is \emph{forbidden} for $H_2$ if $\chi_{r}(H_2\cup\{e\})>\chi_{r}(H_2)$.

We distinguish two types of admissible edges. A directed edge $e$ is \emph{admissible of type I} for $G$ if $e$ and the edges in the edge
set of $G$ have the same direction. A directed edge $e$ is \emph{admissible of type II} if $e$ and the edges in the edge set
of $G$ have opposite direction.   Note that admissible edges of type II allow for edges with opposite directions between two vertices.

For example, Figure \ref{figure1} shows the graph
$\overrightarrow{\mathcal{G}}_{4}:=\overrightarrow{P}_{2^4-1} \cup H(\overrightarrow{\mathcal{G}}_{4})$
where $H(\overrightarrow{\mathcal{G}}_{4})$ is the graph formed with all admissible edges of type I for
$\overrightarrow{P}_{2^4-1}$. In Figure \ref{figure2} we show the graph
$\overleftarrow{\mathcal{G}}_{4}:=\overrightarrow{P}_{2^4-1} \cup H(\overleftarrow{\mathcal{G}}_{4})$ where $H(\overleftarrow{\mathcal{G}}_{4})$ is the graph formed
with all admissible edges of type II for $\overrightarrow{P}_{2^4-1}$.  Since $\overrightarrow{P}_{2^4-1}$ gives rise
to both graphs $\overrightarrow{\mathcal{G}}_{4}$ and $\overleftarrow{\mathcal{G}}_{4}$, they have the same set of vertices $V= \{v_1, \ldots, v_{15} \}$, where $v_1$ is
leftmost vertex and $v_{15}$ is the rightmost vertex. The numbers on the graphs represent the labelings. That is,
\begin{table} [!h]
\small
\begin{center}
    \begin{tabular}{llllll}
        $f(v_1)=1$;   \quad  & $f(v_2)=2$;  \quad & $f(v_3)=1$;   \quad   & $f(v_4)=3$;   \quad  & $f(v_5)=1$;      \\
        $f(v_6)=2$;          & $f(v_7)=1$;        & $f(v_8)=4$;           &$f(v_9)=1$;           & $f(v_{10}=2$;    \\
        $f(v_{11})=1$;       & $f(v_{12})=3$;     &$f(v_{13})=1$;         & $f(v_{14})=2$;       & $f(v_{15})=1$ .  \\
    \end{tabular}
\end{center}
\end{table}

The largest label in  $\overrightarrow{P}_{2^4-1}$ is $4$. So, $4$ is the rank number of the graphs $\overrightarrow{P}_{2^4-1}$,
$\overrightarrow{\mathcal{G}}_{4}$, and $\overleftarrow{\mathcal{G}}_{4}$. Thus, $\chi_r(\overrightarrow{\mathcal{G}}_{4})= \chi_r(\overleftarrow{\mathcal{G}}_{4})=\chi_r(\overrightarrow{P}_{2^4-1})=4$.

Bodlaender et al. \cite{bodlaender} determined the rank number of a path $P_{\eta}$ to be $\left\lfloor \log_{2} \eta \right\rfloor
+1$. Bruoth and M. Hor\v{n}\'{a}k \cite{bruoth} found a similar value, $\left\lfloor\log_2(\eta-1)\right\rfloor+2$, for $C_{\eta}$. It is easy to
see that these two properties are true for both $\chi_{r}(\overrightarrow{P}_{\eta})$ and $\chi_{r}(\overrightarrow{C}_{\eta})$, respectively.
A minimal ranking for these two types of graphs  can be obtained by labeling the vertices $\{v_{i}$ $|$ $1\leq i\leq \eta\}$
with $1, \ldots ,\alpha+1$ where $2^{\alpha}$ is the highest power of $2$ that divides $i$. If $\eta=2^{\alpha}-1$, then the minimal ranking of
$\chi_{r}(\overrightarrow{P}_{\eta})$ is unique. If $\eta=2^{\alpha}$, then the minimal ranking of $\overrightarrow{C}_{\eta}$ is unique.
These two rankings are called the \textit{standard minimal rankings}.  The following lemma summarizes these properties.

\begin{lemma}[\cite{bodlaender,bruoth}] \label{rankingofapath}
 If $\Gamma$ is any of the graphs $\overrightarrow{P}_{2^k-1}$
 or $\overrightarrow{C}_{2^k}$, then $\Gamma$ has a unique minimal ranking and

\[  \chi_k(\Gamma) =
        \left\{
            \begin{array}{ll}
                k, & \hbox{if  \; $\Gamma=\overrightarrow{P}_{2^k-1}$ }\\
                k+1, & \hbox{if \; $ \Gamma=\overrightarrow{C}_{2^k}$.}
            \end{array}
        \right.
\]
\end{lemma}

\begin{figure} 
    \begin{center}
        \includegraphics[width=160mm]{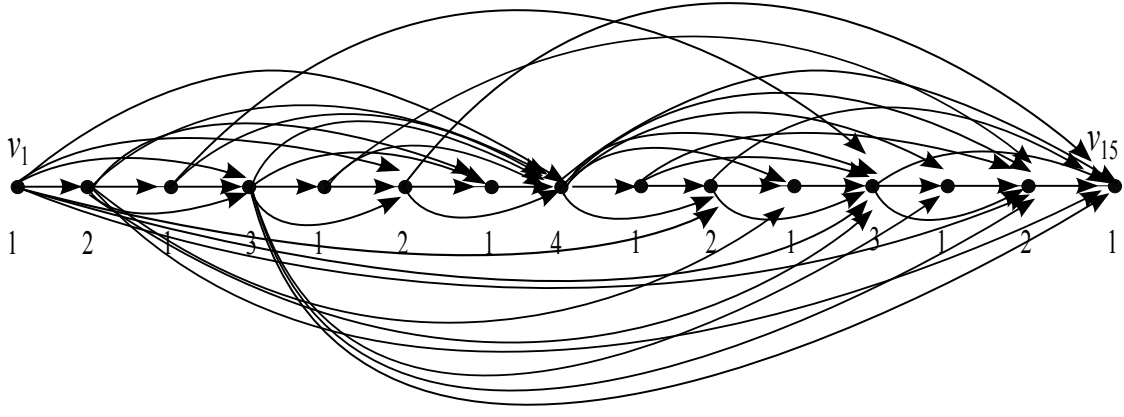}
    \end{center}
    \caption{A graph with admissible edges of type I.}
    \label{figure1}
\end{figure}

\begin{figure} 
    \begin{center}
        \includegraphics[width=160mm]{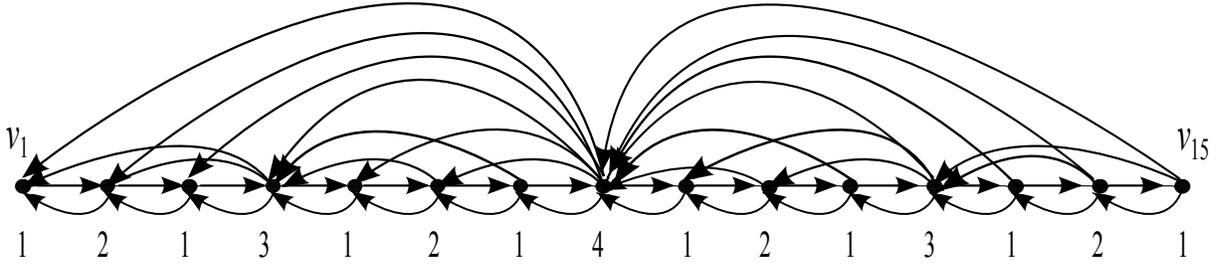}
    \end{center}
    \caption{A graph with admissible edges of type II.} \label{figure2}
\end{figure}

We now give some definitions that are going to be used in the results that follow. Let $G$ be a directed graph with
$f$ as its $k$-ranking function. We define $A_j= \{v \in V(G)\mid f(v)\ge j\}$ and use $\mathcal{C} (A_j)$ to
denote the set of all components of $G \setminus A_j$ where $0<j\le \chi_k(\Gamma)$.

Let $\overrightarrow{\mathcal{G}}_2:= \overrightarrow{P}_{3}$ with $V(\overrightarrow{\mathcal{G}}_2)=\{v_1, v_2, v_{3}\}$, and let
$\overrightarrow{\mathcal{G}}_2^{'}:= \overrightarrow{P^{'}}_{3}$ with $V(\overrightarrow{\mathcal{G}}_2^{'})=\{w_1, w_2, w_{3}\}$.
We define the \emph{direct sum graph of type I}  recursively, denoted by
$\overrightarrow{\mathcal{G}}_{i+1} =\overrightarrow{\mathcal{G}}_i \oplus \overrightarrow{\mathcal{G}}_i^{'}$,  as the graph with vertex set
\[
    V(\overrightarrow{\mathcal{G}}_{i+1})=V(\overrightarrow{\mathcal{G}}_i) \cup V(\overrightarrow{\mathcal{G}}_i^{'}) \cup \{v_{2^{i}}\}
\]
 where $V(\overrightarrow{\mathcal{G}}_i) =\{v_1, v_2, \ldots, v_{2^i-1}\}$ and  $\overrightarrow{\mathcal{G}}_i^{'}$ is the
 graph obtained from a copy of $\overrightarrow{\mathcal{G}}_i$ by  relabeling the vertices of $\overrightarrow{\mathcal{G}}_i$
 as follows $w_t:= v_t$ for $t=1, \ldots, 2^i-1$. That is,
 $V(\overrightarrow{\mathcal{G}}_i^{'}) =\{w_1, \ldots, w_{2^i-1}\}$. The edge set of  $\overrightarrow{\mathcal{G}}_{i+1}$ is
\[
    E(\overrightarrow{\mathcal{G}}_{i+1})=E(\overrightarrow{\mathcal{G}}_i) \cup E(\overrightarrow{\mathcal{G}}_i^{'})\cup E(\ra{H}_{i})
    \cup \{\overrightarrow{v_{2^i-1}v_{2^{i}}}, \overrightarrow{v_{2^{i}}w_1} \},
\]
where
\begin{equation}\label{amisibles:edges:Hi}
\ra{H}_{i}= \{\overrightarrow{v_jv_{2^{i}}}| 1 \leq j \le 2^i-2 \}
\cup \{\overrightarrow{v_{2^{i}}w_j} | 2\le  j \leq 2^i-1\}
 \cup \{\overrightarrow{v_kw_j}| \overrightarrow{v_kv_j} \in E(\overrightarrow{\mathcal{G}}_i)\}.
 \end{equation}
For example, $\overrightarrow{\mathcal{G}}_4$ is depicted in Figure \ref{figure1}.

We use $H(\overrightarrow{\mathcal{G}}_{t})$ to denote the graph with the set of vertices equal to the set of vertices of $\overrightarrow{\mathcal{G}}_{t}$
and the set of edges defined by
\begin{equation}\label{amisibles:edges:Pn}
E(H(\overrightarrow{\mathcal{G}}_{t})):=\bigcup_{j=3}^{t}\ra{H}_{j}.
\end{equation}

\section{Admissible Edges of Type I } \label{sec:admissible_edges1}

This section proves several properties of the direct sum graph of type I. Assume that all admissible edges and direct sum graphs considered throughout
this section are of type I.  The path $\overrightarrow{P}_{2^n-1}$  gives rise to the direct sum graph $\overrightarrow{\mathcal{G}}_{n}$, and they both preserve many
properties, including the symmetry described below.  We prove that $\overrightarrow{\mathcal{G}}_{n}$ is the largest graph that shares with $\overrightarrow{P}_{2^n-1}$ the
set of vertices, the orientation, the $n$-rank number, and the same vertex labeling when the labeling is minimum.  Note that
$\overrightarrow{\mathcal{G}}_{n}$ is $\overrightarrow{P}_{2^{n}-1} \cup H(\overrightarrow{\mathcal{G}}_{n-1})$. We prove that  $H(\overrightarrow{\mathcal{G}}_{n-1})$
is the set of admissible edges for $\overrightarrow{P}_{2^{n}-1}$.

Consider the symmetry seen on the graph $\la{\mathcal{G}}_4$, (see Figure \ref{figure2}).  This symmetry occurs in general in $\overleftarrow{\mathcal{G}}_n$
(see Section \ref{sec:admissible_edges2}).  However, the symmetry in the graph $\overrightarrow{\mathcal{G}}_4$ is not obvious but does exist (see Figure \ref{figure1}).
In Section \ref{sec:matrices}, we give a recursive algorithm to build the adjacency matrix that represents $\overrightarrow{\mathcal{G}}_{n}$.  The adjacency matrix given by
Algorithm \ref{alg:make_matrix} is symmetric with respect to the antidiagonal (for example, see the matrices in Table \ref{example_matrices}).  This symmetry can be predicted
based on how the direct sum graph $\ra{\mathcal{G}}_n$ is defined.

Recall that the \emph{Stirling numbers of the second kind} $S (n,k)$ count the ways to divide a set of $n$ objects into $k$ nonempty subsets where $n, k  \ge 1$.  We are interested
in the following Stirling numbers, $S(n+1,3)=(1/6)(3^{n+1}-3(2^{n+1})+3)$ and $S(n,2)=2^{n-1}-1$.  We prove that the number of edges of $\overrightarrow{\mathcal{G}}_{n}$ and the
number of admissible edges for $\ra{P}_{2^n-1}$ can be described by Stirling numbers of the second kind.

\begin{proposition} \label{lemma:recursiveproperty}  If  $n\geq 2$ and $\overrightarrow{\mathcal{G}}_{n}$ is the direct sum graph of type I, then

\begin{enumerate}
\item $\chi_{r}(\overrightarrow{\mathcal{G}}_{n})=\chi_{r}(\overrightarrow{P}_{2^{n}-1}) =n$ and the minimum labeling of $\overrightarrow{\mathcal{G}}_{n}$ is unique,

 \item  the total number of edges in $\overrightarrow{\mathcal{G}}_{n}$ is $2S(n+1,3)=3^{n}-2^{n+1}+1$,

 \item an edge $\overrightarrow{e}$ is admissible of type I for $\overrightarrow{P}_{2^n-1}$ if and and only if
 $\overrightarrow{e} \in H(\overrightarrow{\mathcal{G}}_{n})$, where  $H(\overrightarrow{\mathcal{G}}_{n})$ is as in (\ref{amisibles:edges:Pn}), and

\item the total number of admissible edges of type I for $\overrightarrow{P}_{2^n-1}$  is $2\left(S(n+1,3)-S(n,2)\right)$.

\end{enumerate}
\end{proposition}

\begin{proof} {\bf Part 1:} The proof is inductive.  Let $T(n)$ be the statement:
$\chi_{r}(\overrightarrow{\mathcal{G}}_{n})=\chi_{r}(\overrightarrow{P}_{2^{n}-1})=n$ for $n>1$ and that $\overrightarrow{\mathcal{G}}_{n}$ and
$\overrightarrow{P}_{2^{n}-1}$ have the same minimal labeling.
For this proof we assume that the labeling of $\overrightarrow{\mathcal{G}}_{n}$ is minimal.

The proof of $T(2)$  is straightforward from the definition of $\overrightarrow{\mathcal{G}}_{2}$. We now suppose that $T(n)$ is true for some fixed $n=k$ with $k>2$. Thus,
$\chi_{r}(\overrightarrow{\mathcal{G}}_{k})=\chi_{r}(\overrightarrow{P}_{2^{k}-1})=k$ is true for some fixed $n=k$ with $k>2$. (We prove $T(k+1)$ is true.)

Consider the graphs $\overrightarrow{P}_{2^k-1}$ and $\overrightarrow{\mathcal{G}}_{k}$, where
$\overrightarrow{\mathcal{G}}_{k}=\overrightarrow{\mathcal{G}}_{k-1} \oplus \overrightarrow{\mathcal{G}}_{k-1}^{'}$.
From the inductive hypothesis and Lemma \ref{rankingofapath} we know that both graphs have $2^{k}-1$ vertices with the same labeling and
that this labeling is minimal and unique. From the definition of  $\overrightarrow{\mathcal{G}}_{k+1}$ we know that its vertices are
$v_1, v_2, \ldots, v_{2^k-1}, v_{2^k}, w_1, w_2, \ldots, w_{2^k-1}$ from left to right. To label $\overrightarrow{\mathcal{G}}_{k+1}$,
we define $f$ as follows: the function $f$ keeps the same labels from $\overrightarrow{\mathcal{G}}_{k}$ for $\{v_1, v_2, \ldots, v_{2^k-1}\}$
and from $\ra{\mathcal{G}}_k^{'}$ for $\{ w_1, w_2, \ldots, w_{2^k-1}\}$ and $f(v_{2^k})=k+1$ since $v_{2^k}$ needs a new label. The function $f$ is a well defined labeling for
$\overrightarrow{\mathcal{G}}_{k+1}$ since $f(v_{2^k}) := k+1$ preserves a good labeling for the edges
\begin{eqnarray*}
    \{\ra{v_iv_{2^k}}, \ra{v_{2^k}w_i} | 1 \leq i \leq 2^k-1\}, 			& &					\{\ra{v_iv_j} | \ra{v_iv_j} \in E(\overrightarrow{\mathcal{G}}_k)\}, \\
    \{\ra{w_iw_j} | \ra{w_iw_j} \in E(\overrightarrow{\mathcal{G}}_k^{'})\},  & \quad \text{and} \quad &\{\ra{v_iw_j} | \ra{v_iv_j} \in E(\overrightarrow{\mathcal{G}}_k)\}.
\end{eqnarray*}
Since one end of each edge in $\{\ra{v_iv_{2^k}}, \ra{v_{2^k}w_i} | 1 \leq i \leq 2^k-1\}$
is labeled with the highest label, these edges are clearly admissible in $\overrightarrow{\mathcal{G}}_{k+1}$.
The edges $\{\ra{v_iv_j} | \ra{v_iv_j} \in E(\overrightarrow{\mathcal{G}}_k)\}$ are admissible in $\overrightarrow{\mathcal{G}}_{k+1}$ since they are admissible in
$\overrightarrow{\mathcal{G}}_k$.  Similarly, the edges $\{\ra{w_iw_j} | \ra{w_iw_j} \in E(\overrightarrow{\mathcal{G}}_k^{'})\}$ are admissible in
$\overrightarrow{\mathcal{G}}_{k+1}$.  The edges $\{\ra{v_iw_j} | \ra{v_iv_j} \in E(\ra{\mathcal{G}}_k\}$ are admissible in $\ra{\mathcal{G}}_{k+1}$ since the subgraph
of $\ra{\mathcal{G}}_{k+1}$ induced by the vertices $\{ v_1, v_2, \ldots, v_i, w_j, w_{j+1}, \ldots, w_{2^k-1} \}$ has the same labeling as the subgraph of $\ra{\mathcal{G}}_k$
induced by $\{v_1, v_2, \ldots, v_i, v_j, v_{j+1}, \ldots, v_{2^k-1}\}$ which has a proper rank labeling.  Note that $f$ is also a minimal labeling for $\overrightarrow{P}_{2^{k+1}-1}$.
This proves $T(k+1)$ is true.

{\bf Part 2:} Let the total number of edges in $\overrightarrow{\mathcal{G}}_{k+1}$ be denoted by $a_{k+1}$.  From the definition of edges of
$\overrightarrow{\mathcal{G}}_{k+1}$ it is easy to see that,
\begin{eqnarray*}
    a_{k+1}&= &|E(\overrightarrow{\mathcal{G}}_{k})| + |E(\overrightarrow{\mathcal{G}}_{k}^{'})| + |\{\ra{v_iv_{2^k}} | 1 \leq i \leq 2^k-2\}| +
                |\{\ra{v_{2^k}w_i} | 2 \leq i \leq 2^k-1\}|  \\
           &  &+\; |\{ \ra{v_iw_j} | \ra{v_iv_j} \in E(\overrightarrow{\mathcal{G}}_k)\}| + |\{\ra{v_{2^k-1}v_{2^k}}, \ra{v_{2^k}w_1}\}| \\
	       &= &a_k + a_k + (2^k-2) + (2^k-2) + a_k + 2 \\
	       &= &3a_k + 2(2^k-1).
\end{eqnarray*}

We prove by induction that the number of edges in $\overrightarrow{\mathcal{G}}_{n}$ is given by $3^n-2^{n+1}+1$ . Let $T(n)$ be the statement:
$a_n=3^n-2^{n+1}+1$ for $n>1$.

We prove $T(2)$. From definition of $\overrightarrow{\mathcal{G}}_{2}$, we have $a_2= 3^2-2^3+1 = 2$. We now suppose that $T(n)$ is true for some fixed $n=k$ with $k>2$.
Therefore, $a_n=3^n-2^{n+1}+1$ is true for some fixed $n=k$ with $k>2$. Thus, $a_{k} = 3^k-2^{k+1}+1$ is true. (We prove $T(k+1)$ is true.)

Since $a_{k+1} = 3a_k + 2(2^k-1)$, we have that
\begin{align*}
a_{k+1} &= 3a_k + 2(2^k-1) \\
	&= 3(3^k-2^{k+1}+1)+2(2^k-1) \\
	&= 3^{k+1}-2^{k+2}+1,
\end{align*}
which shows that $T(k+1)$ holds.  Note that $a_n=3^{n}-2^{n+1}+1$  is twice the Stirling number of second kind. That is,
$\overrightarrow{\mathcal{G}}_{n}$ has $2S(n+1,3)=3^{n}-2^{n+1}+1$ edges.

{\bf Part 3:}  Suppose that  $\overrightarrow{e} \in H(\overrightarrow{\mathcal{G}}_{n})$. Then by the definition of
$E(H(\overrightarrow{\mathcal{G}}_{n}))$ in (\ref{amisibles:edges:Pn}) and part 1 of this proposition,  $\overrightarrow{e}$ is
an admissible edge for $\overrightarrow{P}_{2^n-1}$.

Now suppose that  $\overrightarrow{e}$ is an admissible edge for $\overrightarrow{P}_{2^n-1}$. We prove that
$\overrightarrow{e} \in H(\overrightarrow{\mathcal{G}}_{n})$.
Let $u$ and $v$ be the vertices of $\overrightarrow{e}$  with $f(v)<f(u)=j$.

{\bf Case 1.}  Suppose that $u$ and $v$ are in the same component $\mathcal{C}$ of $\overrightarrow{\mathcal{G}}_{n}\setminus A_{j+1}$.
We prove the case in which $\overrightarrow{e}=\overrightarrow{vu}$ and since the proof of the case in which
$\overrightarrow{e}=\overrightarrow{uv}$ is similar, we omit it. Recall that $\mathcal{C}$ has $2^{j}-1$ vertices and that $j$ is the
largest label in $\mathcal{C}$. Since $u$ has the largest label in the component, it corresponds to the vertex in position $2^{j-1}$. Then
by the definition of $\overrightarrow{e}$, the edge $\overrightarrow{e}$ belongs to $\ra{H}_{j}$ as defined in (\ref{amisibles:edges:Hi}).
Thus, $\overrightarrow{e} \in H(\overrightarrow{\mathcal{G}}_{n})$.

{\bf Case 2.}  Suppose that $u \in \mathcal{C}$ and $v \in \mathcal{C}'$ where $\mathcal{C}$ and $\mathcal{C}'$ are distinct components
of $\overrightarrow{\mathcal{G}}_{n}\setminus A_{j+1}$. Let $w$ be a vertex in
$\mathcal{C}'$ with $f(w)=j$, where $j$ is the largest label in $\mathcal{C}'$.

 {\bf Subcase 1.}  Suppose that $\overrightarrow{e}=\overrightarrow{uv}$. Note that if $v=w$ or if $v$ is located in $\mathcal{C}'$ in a
 position to the left of $w$, then $\overrightarrow{e}$ gives rise to a path connecting $u$ and $w$ which does not contain a larger label in between
 $u$ and $w$. That is a contradiction because $\overrightarrow{e}$ is an admissible edge for $\overrightarrow{P}_{2^n-1}$. Thus, $v$ must be
 located in $\mathcal{C}'$ in a position to the right of $w$. This implies that $\overrightarrow{e}$ satisfies the condition described in the last
 set in the definition of $\ra{H}_{j}$ in (\ref{amisibles:edges:Hi}).

 {\bf Subcase 2.}  Suppose that $\overrightarrow{e}=\overrightarrow{vu}$. Note that if $v=w$ or if $v$ is located in $\mathcal{C}'$ in a
 position to the right of $w$, then $\overrightarrow{e}$ gives rise to a path connecting $w$ and $u$ which does not contain a larger label in between
 $u$ and $w$. That is a contradiction because $\overrightarrow{e}$ is an admissible edge for $\overrightarrow{P}_{2^n-1}$. Thus, $v$ must be
 located in $\mathcal{C}'$ in a position to the left of $w$. This implies that $\overrightarrow{e}$ satisfies the condition described in the last
 set in the definition of $\ra{H}_{j}$ in (\ref{amisibles:edges:Hi}).

{\bf Part 4:} It is easy to see that $\overrightarrow{P}_{2^n-1}$ has $2^n-2$ edges, which can be rewritten as $2(2^{n-1}-1)=2S(n,2)$.
From part 3 of this proposition we know that the set of admissible edges for $\overrightarrow{P}_{2^n-1}$ is $H(\overrightarrow{\mathcal{G}}_{n})$.
Therefore, the total number of admissible edges is the number of edges in $\overrightarrow{\mathcal{G}}_{n}$
minus the number of edges in $\overrightarrow{P}_{2^n-1}$ which is $2(S(n+1,3)-S(n,2))$.
\end{proof}

\section{Admissible Edges of Type II } \label{sec:admissible_edges2}

This section discusses the admissible edges of type II for $\ra{P}_{2^n-1}$ using the direct sum graph of type II defined below.  Throughout this section, it can be
assumed that all admissible edges and direct sum graphs are of type II.  We prove that the direct sum graph is a maximum graph such that $\ra{P}_{2^n-1}$ and
$\la{\mathcal{G}}_n$ have the same set of vertices, $n$-rank number, and vertex labeling when the labeling is minimum.  The symmetry of $\la{\mathcal{G}}_n$ is straightforward, and an example can be seen in Figure \ref{figure2}. Section \ref{sec:matrices} gives a recursive algorithm for the adjacency matrix for
$\la{\mathcal{G}}_n$.  The symmetry found in $\la{\mathcal{G}}_n$ can be clearly seen in the adjacency matrix. It is also anti-diagonal symmetry.

We first discuss several necessary definitions.  Let $\overleftarrow{\mathcal{G}}_2:= \overrightarrow{P}_{3} \cup E(\la{H}_2)$ with vertices
$V(\overleftarrow{\mathcal{G}}_2)=\{v_1, v_2, v_{3}\}$ and edges $E(\la{H}_2)=\{\la{v_1v_2}, \la{v_2v_3} \}$, and let
$\overleftarrow{\mathcal{G}}_2^{'}:= \overrightarrow{P^{'}}_{3} \cup E(\la{H}_2^{'})$ with vertices $V(\overleftarrow{\mathcal{G}}_2^{'})=\{w_1, w_2, w_{3}\}$ and edges
$E(\la{H}_2^{'})=\{\la{w_1w_2}, \la{w_2w_3} \}$. We define the \emph{direct sum graph of type II} recursively, denoted by
$\overleftarrow{\mathcal{G}}_{i+1}=\overleftarrow{\mathcal{G}}_i \oplus \overleftarrow{\mathcal{G}}_i^{'}$,  as the graph with vertex set
$$V(\overleftarrow{\mathcal{G}}_{i+1})=V(\overleftarrow{\mathcal{G}}_i) \cup V(\overleftarrow{\mathcal{G}}_i^{'}) \cup \{v_{2^{i}}\}$$
where $V(\overleftarrow{\mathcal{G}}_i) =\{v_1, v_2, \ldots, v_{2^i-1}\}$ and  $\overleftarrow{\mathcal{G}}_i^{'}$ is the graph obtained from a copy of $\overleftarrow{\mathcal{G}}_i$
by relabeling the vertices of $\overleftarrow{\mathcal{G}}_i$ as follows $w_t:= v_t$ for $t=1, 2, \ldots, 2^i-1$. That is,
$V(\overleftarrow{\mathcal{G}}_i^{'}) =\{w_1, w_2, \ldots, w_{2^i-1}\}$. The edge set of  $\overleftarrow{\mathcal{G}}_{i+1}$ is
\[
    E(\overleftarrow{\mathcal{G}}_{i+1})=E(\overleftarrow{\mathcal{G}}_i) \cup E(\overleftarrow{\mathcal{G}}_i^{'})\cup E(\la{H}_{i})
    \cup \{\ra{v_{2^i-1}v_{2^{i}}}, \ra{v_{2^{i}}w_1} \},
\]
where
\begin{equation}\label{eq:Hileft}
\la{H}_{i}= \{\la{v_jv_{2^i}} | 1 \leq j \leq 2^i-1 \} \cup \{\la{v_{2^{i}}w_j} | 1\leq  j \leq 2^i-1\}.
\end{equation}
For example, $\overleftarrow{\mathcal{G}}_4$ is depicted in Figure \ref{figure2}.

We use $H(\overleftarrow{\mathcal{G}}_{t})$ to denote the graph with the set of vertices equal to the set of vertices of $\overleftarrow{\mathcal{G}}_{t}$
and the set of edges defined by
\begin{equation}\label{amisibles:edges:Pn:2}
E(H(\overleftarrow{\mathcal{G}}_{t})):=\bigcup_{j=2}^{t}\la{H}_{j}.
\end{equation}

Note that when the direction is removed from the edges in this construction, the resulting graph is the graph found in \cite{floreznarayan} for the
undirected case.  Also note that this construction of $\overleftarrow{\mathcal{G}}_n$ is not unique.  That is, there is more than one way to add
admissible edges of type II to $P_{2^n-1}$ to create a graph with the maximum number of these admissible edges as possible while maintaining the rank number.

The numerical sequences in Proposition \ref{lemma:recursiveproperty2} parts (2) and (4) are in Sloane \cite{sloane} at \seqnum{A058922} and \seqnum{A036799}, respectively.

\begin{proposition} \label{lemma:recursiveproperty2}  If  $n\ge 2$ and $\overleftarrow{\mathcal{G}}_{n}$ is the direct sum graph of type II, then

\begin{enumerate}
\item $\chi_{r}(\overleftarrow{\mathcal{G}}_{n})=\chi_{r}(\overrightarrow{P}_{2^{n}-1}) =n$ and the minimum labeling of $\overleftarrow{\mathcal{G}}_{n}$ is unique,

 \item  the total number of edges in $\overleftarrow{\mathcal{G}}_{n}$ is $(n-1)2^n$,

 \item an edge $\overleftarrow{e}$ is admissible of type II for $P_{2^n-1}$ if and and only if $\overleftarrow{e} \in H(\overleftarrow{\mathcal{G}}_{n})$, where
 $H(\overleftarrow{\mathcal{G}}_{n})$ is as in (\ref{amisibles:edges:Pn:2}), and

\item the total number of admissible edges of type II for $\overrightarrow{P}_{2^n-1}$  is $(n-2)2^n+2$.

\end{enumerate}
\end{proposition}

\begin{proof} {\bf Part 1:} We prove this part by induction.  Let $T(n)$ be the statement:
$\chi_{r}(\overleftarrow{\mathcal{G}}_{n})=\chi_{r}(\overrightarrow{P}_{2^{n}-1})=n$ for $n>1$ and that $\overleftarrow{\mathcal{G}}_{n}$
and $\overrightarrow{P}_{2^{n}-1}$ have the same minimal labeling.  For this proof suppose that the labeling of $\overleftarrow{\mathcal{G}}_{n}$ is minimal.

The proof of $T(2)$  is straightforward from the definition of $\overleftarrow{\mathcal{G}}_{2}$. Suppose that $T(n)$ is true for some fixed $n=k$ with $k>2$.
That is, suppose that
$\chi_{r}(\overleftarrow{\mathcal{G}}_{k})=\chi_{r}(\overrightarrow{P}_{2^{k}-1})=k$ is true for some fixed $n=k$ with $k>2$, and we prove $T(k+1)$ is also true.

Consider the graphs $\overrightarrow{P}_{2^k-1}$ and $\overleftarrow{\mathcal{G}}_{k}=\overleftarrow{\mathcal{G}}_{k-1} \oplus \overleftarrow{\mathcal{G}}_{k-1}^{'}$.
From the inductive hypothesis and Lemma \ref{rankingofapath}
we know that both graphs have $2^{k}-1$ vertices with the same labeling and that it is minimal and unique.
From the definition of  $\overleftarrow{\mathcal{G}}_{k+1}$ we know that its vertices are $v_1, v_2, \ldots, v_{2^k-1}, v_{2^k}, w_1, w_2, \ldots, w_{2^k-1}$
from left to right. To label $\overleftarrow{\mathcal{G}}_{k+1}$, we define $f$ as follows: the function $f$ keeps the same labels from
$\overleftarrow{\mathcal{G}}_{k}$ for $\{v_1, v_2, \ldots, v_{2^k-1}\}$ and from $\la{\mathcal{G}}_k^{'}$ for $\{ w_1, w_2, \ldots, w_{2^k-1}\}$ and $f(v_{2^k})=k+1$ since
$v_{2^k}$ needs a new label. The function $f$ is a well-defined labeling for $\overleftarrow{\mathcal{G}}_{k+1}$ since $f(v_{2^k}) := k+1$ preserves a proper
labeling for the edges
\begin{eqnarray*}
\{
    \la{v_iv_{2^k}}, \la{v_{2^k}w_i} | 1 \leq i \leq 2^k-1\}, 				   &  & \{\la{v_iv_j} | \la{v_iv_j} \in E(\overleftarrow{\mathcal{G}}_k)\},\\
    \quad  \{\la{w_iw_j} | \la{w_iw_j} \in E(\overleftarrow{\mathcal{G}}_k^{'})\}, 	&\quad \text{and} \quad  & \{\ra{v_{2^k-1}v_{2^k}}, \ra{v_{2^k}w_1}
\}.
\end{eqnarray*}
Since one end of each edge in $\{\la{v_iv_{2^k}}, \la{v_{2^k}w_i} | 1 \leq i \leq 2^k-1\}$ and $\{\ra{v_{2^k-1}v_{2^k}}, \ra{v_{2^k}w_1}\}$
is labeled with the highest label, it is clear that these edges are admissible for $\overleftarrow{\mathcal{G}}_{k+1}$.
The edges $\{\la{v_iv_j} | \la{v_iv_j} \in E(\overleftarrow{\mathcal{G}}_k)\}$ are admissible in $\overleftarrow{\mathcal{G}}_{k+1}$ since they are admissible in
$\overleftarrow{\mathcal{G}}_k$.  Similarly, the edges $\{\la{w_iw_j} | \la{w_iw_j} \in E(\overleftarrow{\mathcal{G}}_k^{'})\}$ are admissible  in
$\overleftarrow{\mathcal{G}}_{k+1}$. Note that $f$ is also a minimal labeling for $\overrightarrow{P}_{2^{k+1}-1}$. This proves $T(k+1)$ is true.

{\bf Part 2:} Let the total number of edges in $\overleftarrow{\mathcal{G}}_{k+1}$ be denoted by $b_{k+1}$.  From the definition of edges of
$\overleftarrow{\mathcal{G}}_{k+1}$ it is easy to see that,
\begin{align*}
    b_{k+1} &= |E(\overleftarrow{\mathcal{G}}_k)| + |E(\overleftarrow{\mathcal{G}}_k^{'})| + |\{ \la{v_iv_{2^k}},
    \la{v_{2^k}w_i} | 1 \leq i \leq 2^k-1 \}| + |\{ \ra{v_{2^k-1}v_{2^k}}, \ra{v_{2^k}w_1}\}| \\
	&= b_k + b_k + 2(2^k-1)+2 \\
	&= 2b_k + 2^{k+1}.
\end{align*}

We prove by induction that the number of edges in $\overleftarrow{\mathcal{G}}_{k+1}$ is given by $k2^{k+1}$. Let $T(n)$ be the statement:
$b_n:=(n-1)2^n$ for $n>1$.

We prove $T(2)$. It is easy to see, from definition of $\overleftarrow{\mathcal{G}}_{2}$, that $b_2= (2-1)2^2 = 4$.
Suppose that $T(n)$ is true for some fix $n=k$ with $k>2$. That is, suppose that $b_{k} =(k-1)2^k$ for some fix $n=k$
with $k>2$ and  we prove $T(k+1)$. Since $b_{k+1} = 2b_k+2^{k+1}$, we have that
\begin{align*}
b_{k+1} &= 2b_k+2^{k+1} \\
	&= 2((k-1)2^k)+2^{k+1} \\
	&= k2^{k+1}.
\end{align*}
Thus, $\overleftarrow{\mathcal{G}}_{n}$ has $(n-1)2^n$ edges.

{\bf Part 3:} Suppose that $\la{e} \in H(\la{\mathcal{G}}_n)$.  Then by the definition of $E(H(\la{\mathcal{G}}_n)$ and part 1 of this
proposition, $\la{e}$ is an admissible edge for $\ra{P}_{2^n-1}$.

Now suppose that $\la{e}$ is an admissible edge for $\ra{P}_{2^n-1}$.  We prove that $\la{e} \in H(\la{\mathcal{G}}_n)$ by induction.  Let $T(n)$
be the statement: if $\la{e}$ is admissible in $\ra{P}_{2^n-1}$, then $\la{e} \in H(\la{\mathcal{G}}_n)$ for $n>1$.

We prove $T(2)$.  Let $\la{e}$ be admissible in $\ra{P}_{2^2-1}$.  Then either $\la{e}$ is in $\la{H}_2$ or is $\la{v_1v_3}$.
The edge $\la{v_1v_3}$ leads to a contradiction since $f(v_1)=f(v_3)=1$ violates a proper labeling.  Therefore, if $\la{e}$ is
admissible in $\ra{P}_{2^2-1}$, then $\la{e} \in \la{H}_2$, and thus $\la{e} \in H(\la{\mathcal{G}}_2)$.

Suppose that $T(n)$ is true for some fixed $n=k$ with $k>2$.  Suppose that $\la{e}$ is admissible in $P_{2^{k+1}-1}$ with
$u$ and $v$ as endpoints such that $f(v)<f(u)=j$.

{\bf Case 1.} Suppose that $u$ and $v$ are in the same component $\mathcal{C}$ of $\la{\mathcal{G}}_{k+1}\setminus A_{j+1}$.  Then $u$ has the
largest label in $\mathcal{C}$ and is in position $2^{j-1}$.  Thus, $\la{e} \in \la{H}_j$ as defined in (\ref{eq:Hileft}) and $\la{e} \in H(\la{\mathcal{G}}_{k+1})$.

{\bf Case 2.} Suppose that $u \in \mathcal{C}$ and $v \in \mathcal{C}^{'}$ where $\mathcal{C}$ and $\mathcal{C}^{'}$ are distinct components of
$\la{\mathcal{G}}_{k+1} \backslash A_{j+1}$.  Let $w \in \mathcal{C}^{'}$ such that $f(w)=j$.  The edge $\la{e}$ gives rise to a path connecting
$u$ and $w$ which does not contain a larger label in between $u$ and $w$ since each component contains all dual direction edges on the path.  Such a
path contradicts $\la{e}$ being admissible in $\ra{P}_{2^{k+1}-1}$. Therefore, $\la{e}$ is admissible in $\ra{P}_{2^n-1}$ if and only if
$\la{e} \in H(\la{\mathcal{G}}_n)$.

{\bf Part 4:} It is easy to see that $\overrightarrow{P}_{2^n-1}$ has $2^n-2$ edges. From part 3 of this proposition we know that set of
admissible edges for $\overrightarrow{P}_{2^n-1}$ is $H(\overleftarrow{\mathcal{G}}_{n})$. Therefore, the total number of admissible edges is the
number of edges in $\overleftarrow{\mathcal{G}}_{n}$ minus the number of edges in
$\overrightarrow{P}_{2^n-1}$ which is $(n-2)2^n+2$.
\end{proof}

\section{Adjacency Matrices of $\protect\overrightarrow{\mathcal{G}}_{n}$ and  $\protect\overleftarrow{\mathcal{G}}_{n}$} \label{sec:matrices}

In this section we give recursive algorithms that highlight the symmetric structure of the graphs $\overrightarrow{\mathcal{G}}_{n}$ and
$\overleftarrow{\mathcal{G}}_{n}$. The algorithms are based on block-recursive adjacency matrices for direct sum graphs of type I and II.
The matrices present symmetry with respect to the antidiagonal rather than the main diagonal (see for example Tables
\ref{example_matrices} and \ref{example_matrices_II}).

In Table \ref{example_matrices} we show  the matrices  $A_{2}$ and $A_{3}$ that represent direct sum graphs of type I. We observe that $A_2$
forms three blocks within $A_3$. Similarly, $A_4$ will contain three blocks of $A_3$ and so on.  As mentioned previously, this symmetry is not
obvious from looking at the corresponding graph, such as in Figure \ref{figure1}. The component symmetry of direct sum graphs of type II is clear
from a graph, such as in Figure \ref{figure2}, but the fact that it is antidiagonal symmetry is obvious in the adjacency matrix found in Table \ref{example_matrices_II}.
It should be clear that matrices for direct sum graphs of type I have the same block-recursive structure as matrices
for direct sum graphs of type II, but the contents of the blocks are different.

In Algorithm \ref{alg:make_matrix}, $A_{k}$ denotes a $(2^{k}-1)\times (2^{k}-1)$ matrix. We use $k$ in this manner because it simplifies our description
of the recursion. We denote by $1_{k}$ the vector of length $2^{k}-1$ where all entries are $1$ and the transpose of this vector is denoted by $1_{k}^T$.
We denote by $0_{k}$ the $2^k\times 2^{k}$ matrix where all its entries are zero. We divide matrix $A_{k}$ into blocks with the layout shown in
Algorithm~\ref{alg:make_matrix}. With this example in mind, our algorithm for constructing a matrix $A_{k}$ follows. Note that $A_{0}$ and $1_{0}$ are not defined.

\begin{table}[htbp]
  \centering
  $A_2= \left[
            \begin{array}{ccc}
                    0&1&0\\
                    0&0&1 \\
                    0&0&0\\
            \end{array}
    \right]
    \hspace{2cm}
A_{3} =\left[
		\begin{array}{cc}
	       \begin{array}{|ccc|c|ccc|} 0&1&0&1&0&1&0 \\ 0&0&1&1&0&0&1 \\  0&0&0&1&0&0&0 \\\cline{1-3}  \cline{5-7}\end{array}\\
		\begin{array}{ccccccc} 0&0&0&0&1&1&1 \\\cline{5-7}  \end{array}\\
		\begin{array}{cccc|ccc|} 0&0&0&0&0&1&0 \\  0&0&0&0&0&0&1 \\  0&0&0&0&0&0&0  \\ \end{array}\\
	\end{array}\right]$
  \caption{The adjacency matrices for the direct sum graphs of type I with $2^2-1$ and $2^3-1$ vertices, respectively.}\label{example_matrices}
\end{table}

We observe, from the recursive definition of the graph $\overrightarrow{\mathcal{G}}_{n}$, that the adjacency matrix of $\overrightarrow{\mathcal{G}}_{n}$
embeds in the Sierpi\'nski sieve triangle. For example, Figure \ref{sierpinski} shows $A_3$ embedded in the corresponding triangle.
The entries of $A_3$ are within the region  bounded by the dashed diamond shape in Figure \ref{sierpinski} part (a).
Thus, the entries on the main diagonal of $A_3$ correspond to the entries on line nine of the Sierpi\'nski sieve triangle (which is row eight within the diamond).
The entries on the super diagonal of $A_3$ correspond to the entries on line eight of the Sierpi\'nski sieve triangle (which is row seven within the diamond).  The
entries on the subdiagonal of $A_3$ correspond to the entries on line 10  of the Sierpi\'nski sieve triangle (which is row nine within the diamond) and so on.
Thus, the entry $a_{ij}$ of $A_k$ is given by ${2^k -j+i \choose i} \mod 2$.
Due to this embedding, $\overrightarrow{\mathcal{G}}_{n}$ should be called the  \emph{directed Sierpi\'nski graph}.
In fact, because of this embedding, we can use any number of algorithms to construct $A_k$ and the corresponding graphs.
For related graphs, see for example the undirected Pascal  graphs defined by Deo and Quinn \cite{deo}. The undirected Sierpi\'nski graph
(or Hanoi graph) is represented in Figure \ref{sierpinski} part (b) (see for example Romik \cite{romik}).

One motivation for Deo and Quinn \cite{deo} in describing undirected Pascal graphs was to define bidirectional
computer network topologies with certain connectivity and cohesion constraints. It should be obvious how the
incidence matrix of a type I graph embeds in an undirected Pascal graph--replacing by zero all nonzero entries that
are below the main diagonal of the adjacency matrix of an undirected Pascal graph. Modern networks enable roles,
communication, protocols, or permissions for which operations on distributed systems are asymmetric. Our directed graphs
share some of the properties of Pascal graphs and may be useful in defining asymmetric computer networks with guaranteed properties.

\begin{figure} [htbp]
\begin{center}
\includegraphics[width=75mm]{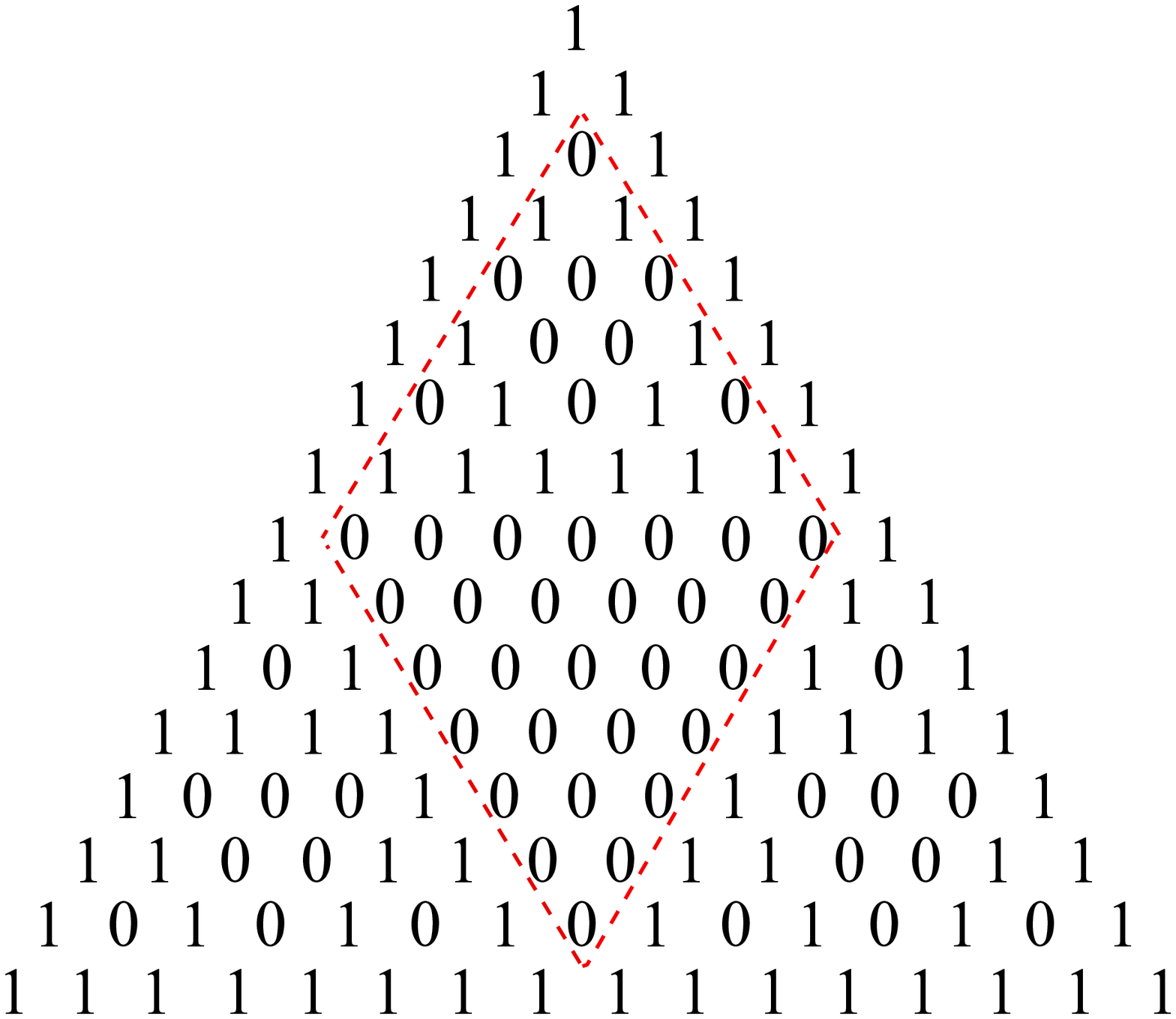} \hspace{1cm}
\includegraphics[width=70mm]{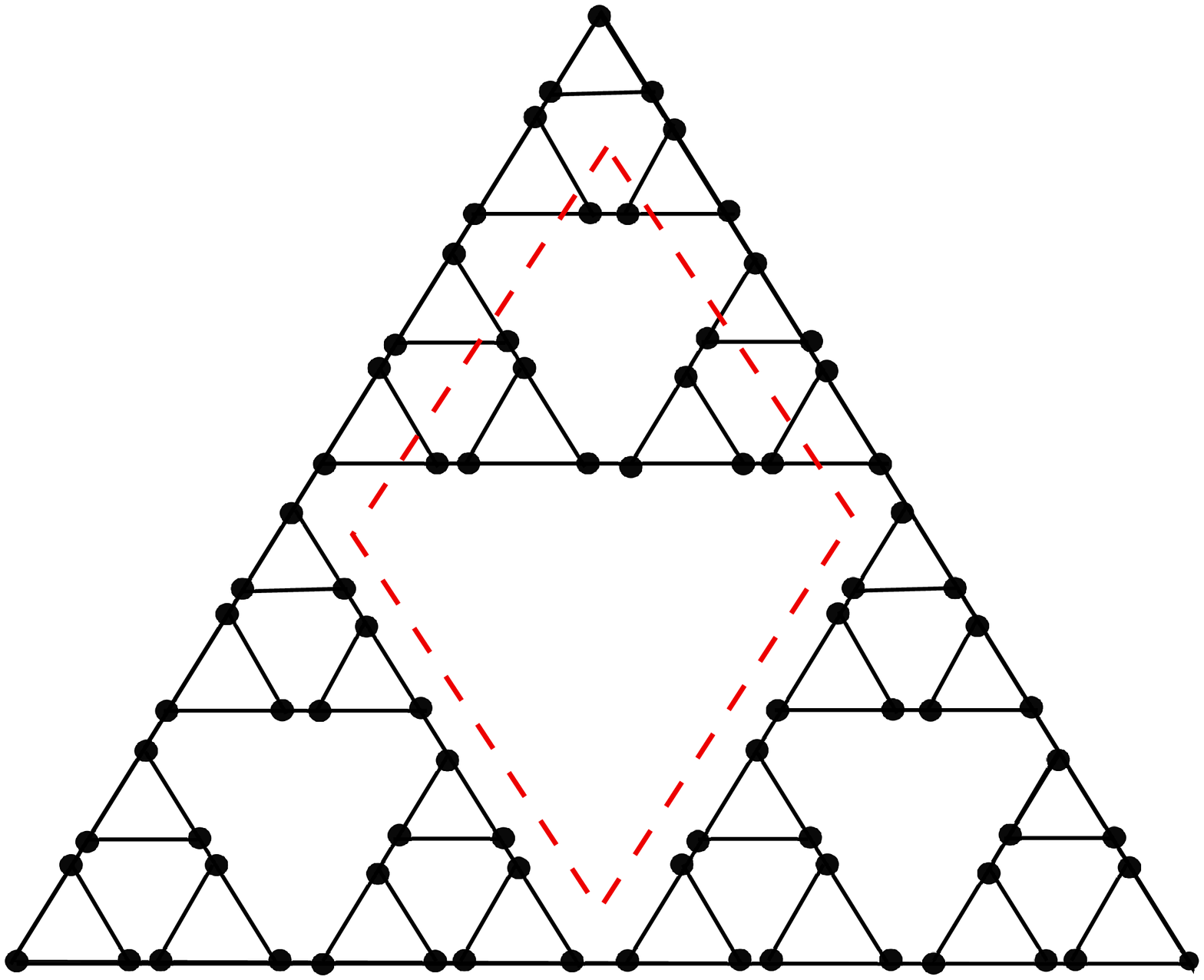}
\end{center}
\caption{(a)\; Sierpi\'nski triangle. \hspace{2cm} (b)\; Sierpi\'nski sieve graph  or Hanoi graph.}
\label{sierpinski}
\end{figure}

In describing Algorithm \ref{alg:make_matrix_II}, we use $B$, $t$ and $u$ to avoid confusion with Algorithm \ref{alg:make_matrix}.
We use $B_{t}$ to denote a $(2^{t}-1)\times (2^{t}-1)$ matrix. As with $k$ above, we use $t$ in this manner because it simplifies
our description of the recursion. We denote by $1_{t}$ the vector of length $2^{t}-1$ where all entries are $1$. We denote by $0_{t}$
the $(2^{t}-1)\times (2^{t}-1)$ matrix where all its entries are zero. We denote by $J_{t}$ a single-entry column vector of length
$(2^{t}-1)$ where the last element is 1 and all others are 0 and by $J'_{t}$ a single-entry row vector where the first element is 1
and all others are 0. We divide matrix $B_{t}$ into blocks with the layout shown in
Algorithm~\ref{alg:make_matrix_II}. For example, in Table \ref{example_matrices_II}, we show the matrices  $B_{2}$ and $B_{3}$.

Our algorithm for constructing $B_{t}$ follows. As with $A_{k}$, $B_{0}$ and $1_{0}$ are not defined.

\begin{table} [htbp]
\begin{center}
$B_2= \left[
            \begin{array}{ccc}
                    0&1&0\\
                    1&0&1 \\
                    0&1&0\\
            \end{array}
    \right]
    \hspace{2cm}
B_{3} =\left[
		\begin{array}{cc}
	       \begin{array}{|ccc|c|ccc|} 0&1&0&0&0&0&0 \\ 1&0&1&0&0&0&0 \\  0&1&0&1&0&0&0 \\\cline{1-3} \cline{5-7} \end{array}\\
		\begin{array}{ccccccc} 1&1&1&0&1&0&0 \\\cline{1-3} \cline{5-7} \end{array}\\
		\begin{array}{|ccc|c|ccc|} 0&0&0&1&0&1&0 \\  0&0&0&1&1&0&1 \\  0&0&0&1&0&1&0  \\ \end{array}\\
	\end{array}\right]$
\end{center}
\caption{The adjacency matrices for the direct sum graphs of type II with $2^2-1$ and $2^3-1$ vertices, respectively.}
\label{example_matrices_II}
\end{table}

\begin{algorithm}
 \begin{algorithmic}
 	\State $1_{k} = \begin{bmatrix}1\end{bmatrix}_{(2^{k}-1)\times 1}$
 	\State $0_{k} = \begin{bmatrix}0\end{bmatrix}_{2^{k}\times 2^{k}}$
	\State $A_{1} = \begin{bmatrix}0 \end{bmatrix}$\\
 	
\Function{buildMatrix}{$k$}
		\If {$k \equiv  1$}
			\State  $A_{k} \gets A_{1}$
		\Else
			\State $j \gets k-1$ \\
			\State $A_{k} \gets \begin{bmatrix}
                   A_{j} & 1_{j} & A_{j} \\
                     \multicolumn{2}{c}{\multirow{2}{*}{$0_{j}$}} & 1_{j}^{T} \\
                     && A_{j} \end{bmatrix}$ \\
		\EndIf
\State \Return $A_{k}$
 \EndFunction
 \end{algorithmic}
  \caption{Construct $A_k$ from $A_{k-1}$ recursively.}
 \label{alg:make_matrix}
 \end{algorithm}

\begin{algorithm}
 \begin{algorithmic}
 	\State $1_{t} = \begin{bmatrix}1\end{bmatrix}_{(2^{t}-1)\times 1}$
 	\State $0_{t} = \begin{bmatrix}0\end{bmatrix}_{(2^{t}-1)\times (2^{t}-1)}$
    \State $J_{t}$ = single-entry vector $J_{(2^{t}-1)\times 1 }$, where the last element is 1, all others are 0
    \State $J'_{t}$ = single-entry vector $J_{1 \times (2^{t}-1) }$, where the first element is 1, all others are 0
	\State $B_{1} = \begin{bmatrix}0 \end{bmatrix}$\\
	
\Function{buildMatrixII}{$t$}
		\If {$t \equiv  1$}
			\State  $B_{t} \gets B_{1}$
		\Else
			\State $u \gets t-1$ \\
			\State $B_{t} \gets \begin{bmatrix}
                    B_{u} & J_{u} & 0_{u} \\
                    1_{u}^{T} &   0   & J'_{u} \\
                    0_{u} & 1_{u} & B_{u} \end{bmatrix} $ \\
		\EndIf
\State \Return $B_{t}$
 \EndFunction
 \end{algorithmic}
  \caption{Construct $B_t$ from $B_{t-1}$ recursively.}
 \label{alg:make_matrix_II}
 \end{algorithm}	
	
\section{Admissible Edges for a Directed Cycle} \label{sec:admissible_directed_cyle}

In this section we use the results from previous sections to find the admissible edges of type I and type II for  $\overrightarrow{C}_{2^k}$ and prove similar results. Figure
\ref{figure1_cycle} part (a) shows $\overrightarrow{C}_{2^k}$ with the admissible edges of type I.  The number of edges in this graph is 65 and the number
of admissible edges is 50. Figure  \ref{figure1_cycle} part (b)  shows $\overrightarrow{C}_{2^k}$ with the admissible edges of type II.   The number of
edges in this graph is 65 and the number of admissible edges is 49. The number of  admissible  edges and the number of edges of the new  graph are given
by known numerical sequences. One of these sequences is the Stirling numbers. In both parts the rank number is equal to the rank number of $\overrightarrow{C}_{2^k}$.

\begin{figure} [htbp]
\begin{center}
\includegraphics[width=70mm]{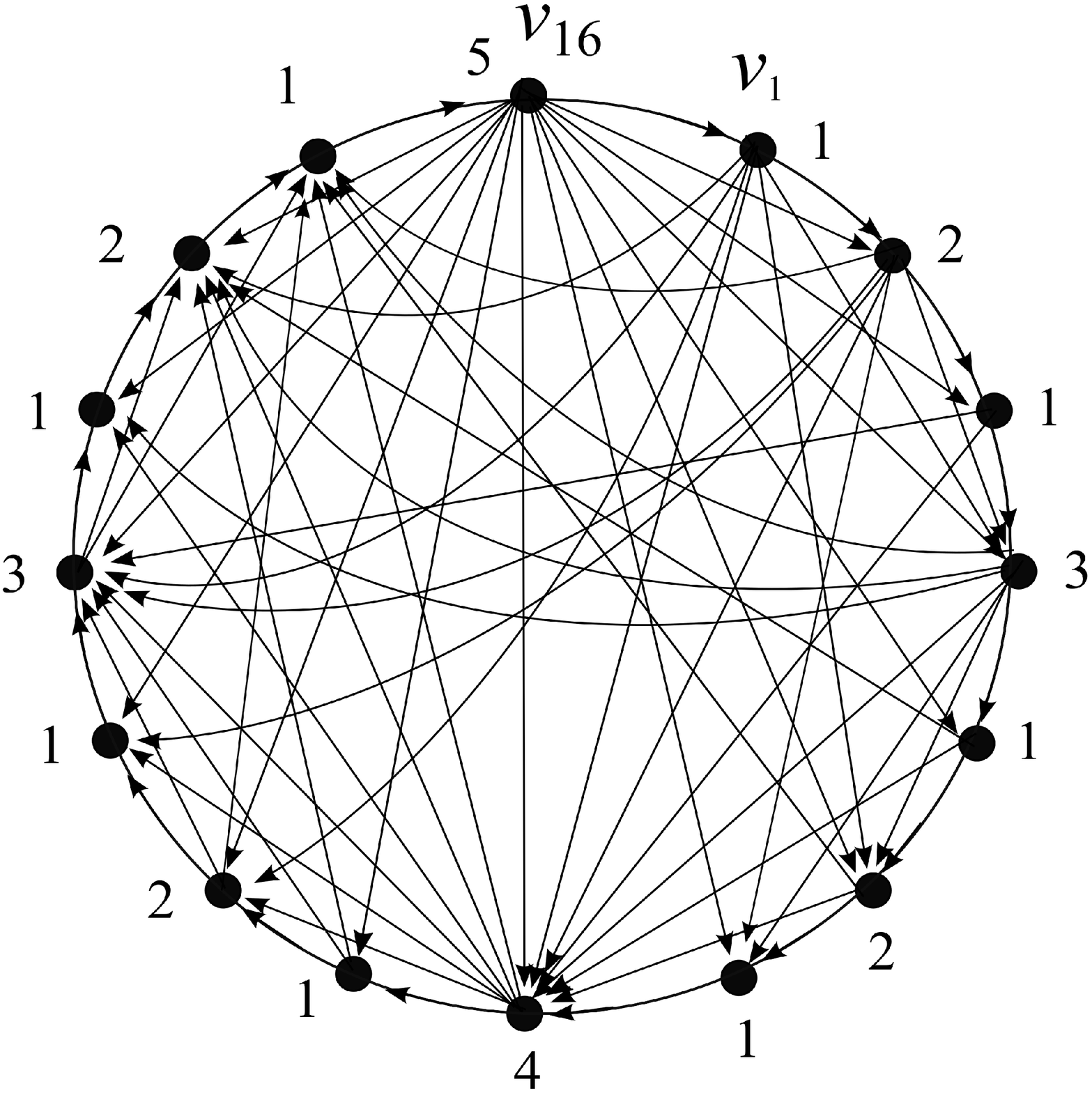} \hspace{1cm}
\includegraphics[width=70mm]{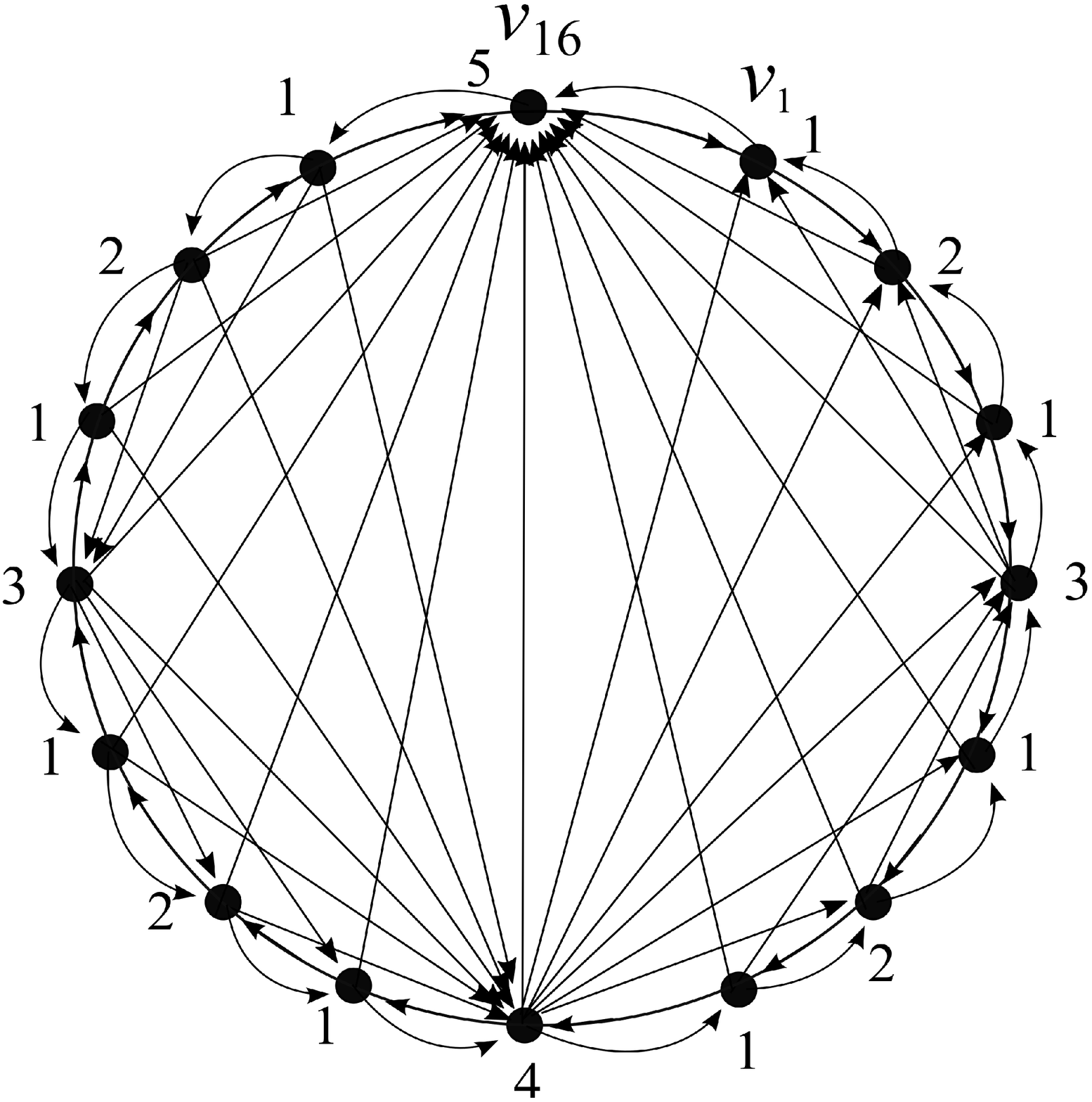}
\end{center}
\caption{(a)\; Admissible edges of type I. \hspace{.5cm} (b)\; Admissible edges of type II.}
\label{figure1_cycle}
\end{figure}

We recall that $V:=\{v_{1},v_{2},\ldots,v_{2^k}\}$ is the set of vertices of $\overrightarrow{C}_{2^k}$, and that $V\setminus \{v_{2^k} \}$
is the set of vertices of  $\overrightarrow{P}_{2^{k}-1}$. We define

\begin{equation}\label{rightHC}
 H_{\overrightarrow{C}} =  H(\overrightarrow{\mathcal{G}}_{k}) \cup \{\overrightarrow{e} \mid \overrightarrow{e} \not\in \overrightarrow{C}_{2^k}
    \text{ with } \ra{e}=\ra{v_{2^k}v_i} \text{ where } i \in \{2, \ldots, 2^k-2 \} \}
\end{equation}

where $H(\overrightarrow{\mathcal{G}}_{k})$ is the set of admissible edges of type I for $\overrightarrow{P}_{2^{k}-1}$  (see (\ref{amisibles:edges:Pn})).
We now define

\begin{equation}\label{leftHC}
 H_{\overleftarrow{C}} = H(\overleftarrow{\mathcal{G}}_{k}) \cup \{\overleftarrow{e} \mid \overleftarrow{e} \not\in \overrightarrow{C}_{2^k}
    \text{ with } \la{e} = \la{v_{2^k}v_i} \text{ where }  i \in \{2, \ldots, 2^k-2 \}  \}
\end{equation}

where $H(\overleftarrow{\mathcal{G}}_{k})$ is the set of admissible edges of type II for $\overrightarrow{P}_{2^{k}-1}$ (see (\ref{amisibles:edges:Pn:2})).

We use $\overrightarrow{\Omega}_{2^{k}}$ and $\overleftarrow{\Omega}_{2^{k}}$ to mean the graphs $\overrightarrow{C}_{2^{k}} \cup H_{\overrightarrow{C}}$
and $\overrightarrow{C}_{2^{k}} \cup H_{\overleftarrow{C}}$, respectively. The numerical sequences in Proposition \ref{goodarcs} parts (3), (4) and (6)
are in Sloane \cite{sloane} at \seqnum{A001047}, \seqnum{A002064}
(called Cullen numbers),  and \seqnum{A048495}, respectively.

\begin{proposition} \label{goodarcs}  If $\overrightarrow{\Omega}_{2^{k}}=\overrightarrow{C}_{2^{k}} \cup H_{\overrightarrow{C}}$ and
 $\overleftarrow{\Omega}_{2^{k}} = \overrightarrow{C}_{2^{k}} \cup H_{\overleftarrow{C}}$ with $k\ge 2$, then
\begin{enumerate}
\item the set $H_{\overrightarrow{C}}$ is the set of admissible  edges of type I for $\overrightarrow{C}_{2^k}$
if and only if
$$\chi_r( \overrightarrow{\Omega}_{2^{k}})= \chi_r(\overrightarrow{C}_{2^k})=k+1,$$

\item the set $ H_{\overleftarrow{C}}$ is the set of admissible  edges of type II  for $\overrightarrow{C}_{2^k}$
if and only if
$$\chi_r(\overleftarrow{\Omega}_{2^{k}})= \chi_r(\overrightarrow{C}_{2^k})=k+1,$$

\item  the total number of edges in $ \overrightarrow{\Omega}_{2^{k}}$ is $3^{k}-2^{k}$,

\item  the total number of edges in $ \overleftarrow{\Omega}_{2^{k}}$ is $k2^k+1$,

\item  the total number of admissible edges of type I for $\overrightarrow{C}_{2^k}$  is the Stirling number of the second kind $2S(k+1,3)=3^{k}-2^{k+1}+1$,
 and

\item the total number of admissible edges of type II for $\overrightarrow{C}_{2^k}$  is $(k-1)2^k+1$.
\end{enumerate}

\end{proposition}

\begin{proof}
We prove parts 1, 3, and 5.  The proofs for part 2, 4, and 6 are similar, respectively, and we omit them.  We assume all admissible edges are of type I throughout this proof.

{\bf Part 1:} For the proof of  necessity, it is easy
to see that if  $H_{\overrightarrow{C}}$ is not a set of admissible edges, then this set contains a
forbidden edge; therefore the rank number of $\overrightarrow{\Omega}_{2^{k}}$ is different that the
rank number of  $\overrightarrow{C}_{2^k}$.  That is a contradiction.

We now prove sufficiency. From Lemma \ref{rankingofapath} we know that
$\chi_r(\overrightarrow{C}_{2^k})=k+1$ and that the minimal ranking is
unique. Suppose that $f:V(\overrightarrow{C}_{2^k})\rightarrow\{1,2,\ldots,k+1\}$
is the ranking function of $\overrightarrow{C}_{2^k}$. So,
$f(v_{2^k})=k+1$.

Let $\overrightarrow{e}_1:=\ra{v_{2^k-1} v_{2^k}}$ and $\overrightarrow{e}_2:=\ra{v_{2^k} v_{1}}$ be two
edges of $\overrightarrow{C}_{2^k}$ and let $H'$ be subgraph of $H_{\overrightarrow{C}}$ formed by all edges of
$H_{\overrightarrow{C}}$ that have vertices in $V'= V\left(H_{\overrightarrow{C}}\right) \setminus \{ v_{2^k} \} = \{ v_1, v_2,
\ldots, v_{2^k-1} \}$. That is, $H'=H(\overrightarrow{\mathcal{G}}_{k})$ and $ V'$ is the set of
vertices of $\overrightarrow{C}_{2^k} \setminus \{\ra{e_1}, \ra{e_2}\}$. Then by Theorem \ref{goodarcs}, we have $E(H')$ is a
set of admissible edges for the graph $\overrightarrow{C}_{2^k} \setminus \{\ra{e_1}, \ra{e_2}\}$ if
and only if $\chi_r(\overrightarrow{C}_{2^k} \setminus \{e_1, e_2\} \cup H') = k$.  The vertices of
$\overrightarrow{C}_{2^k} \setminus \{e_1, e_2\} \cup H'$ have the same labels as the vertices $V'$.

We now prove that $\chi_r(\overrightarrow{C}_{2^k} \cup H_{\overrightarrow{C}})=k+1 $. Let $\overrightarrow{e}$ be an edge
in $H_{\overrightarrow{C}}\setminus H'$. That is,
\[
    \overrightarrow{e} \in \left\{\overrightarrow{e} \mid \overrightarrow{e} \not\in \overrightarrow{C}_{2^k}
    \text{ with } \ra{v_{2^k} v_i} \text{ where } i \in \{2, \ldots, 2^k-2 \} \right\}.
\]
Therefore, the vertices of $\overrightarrow{e}$ are $v_{2^k}$ and
$v_i$ for some $2 \le i \leq 2^k-2$. From definition of the labeling
function $f$ we know that $f(v_{2^k})=k+1$ and $f(v_i)< k+1$. We
do not create any new paths in $\ra{\Omega}_{2^k}$ connecting vertices with label
$k+1$. This implies that the number of labels does not increase. This proves part 1.

{\bf Part 3:} We consider the sets of edges $W$ and $W'$ defined as follows:
\[ W:=\{\ra{v_{2^{k}}v_i}| i= 2, \ldots, 2^{k}-2\} \text{ and } W':= \{ \ra{v_{2^{k}}v_{1}} ,\ra{v_{2^{k}-1}v_{2^k}} \}.\]
The cardinality of $W$ is $2^k-3$. From Proposition \ref{lemma:recursiveproperty}  part 1 we know that all
admissible edges for $\overrightarrow{P}_{2^k-1}$ are also admissible for $\overrightarrow{C}_{2^k}$.
Therefore, the maximum number of edges that can be added to $\overrightarrow{C}_{2^k}$ without changing
its rank number is equal to maximum number of edges that can be added to $\overrightarrow{P}_{2^k-1}$ plus all edges in $W$. Thus,
the total number of edges in $ \overrightarrow{\Omega}_{2^{k}}$ is equal to the number of admissible edges for
$\overrightarrow{P}_{{2^k}-1}$ and the number of edges in  $\overrightarrow{P}_{{2^k}-1}$, plus the
number of edges in $W\cup W'$.  Therefore, the number of edges in $\overrightarrow{\Omega}_{2^{k}}$ is $3^{k}-2^{k}$.

{\bf Part 5:} This proof is straightforward by counting the number of edges that are admissible for
$\overrightarrow{P}_{2^k-1}$ and adding the number of edges in $W$.

\end{proof}

\begin{corollary} \label{uniqueranks}  The graphs $\overrightarrow{\mathcal{G}}_{k}$,
$\overleftarrow{\mathcal{G}}_{k}$, $\overrightarrow{\Omega}_{2^{k}}$, $\overleftarrow{\Omega}_{2^{k}}$ have unique minimal rankings.
\end{corollary}

\section{Admissible Graphs for Directed Paths and Cycles} \label{sec:admissible_graphs}

In this section, we explore constructing new graphs by attaching directed paths and directed cycles to the direct sum graphs and the
omega graphs built in the previous sections. We give algorithms for labeling the new resulting graphs. The algorithms keep the same
rank number as the original graph. Thus, the rank number of the graphs constructed here is either the rank number of  $\overrightarrow{P}_{2^{k}-1}$ or of
$\overrightarrow{C}_{2^k}$.

Finding the rank number of a given graph is a hard problem, even for simple graphs. In the previous sections, we took a known graph
with known rank number, and we built a new graph that  preserves the rank number and as well the set of vertices. In this section, we
explore the same idea, but without preserving the set of vertices. Thus, we give some results on how to build new graphs from a base
graph such that the new graph is larger than the original in terms of the number of vertices and preserves the rank number of the base graph.

Recall from Section \ref{sec:preliminaries} that $\overrightarrow{P}_{2^k-1}$ or $\overrightarrow{C_{2^k}}$ have vertex sets
$V= \{v_1, v_2, \ldots, v_{2^k-1}\}$ and $V \cup \{v_{2^k}\}$, respectively. We construct a new graph by attaching a directed path or a directed cycle
to the vertex $v_i$ for some $i$. Let $\{w_1, w_2, \ldots, w_j=v_i\}$ be the edges of a directed path of length $j$ that is attached to
$\overrightarrow{P}_{2^k-1}$ or $\overrightarrow{C_{2^k}}$ at the vertex $v_i$. The path is denoted by $\overrightarrow{P_{j}}^i$ if its edges are
directed as  $w_{l}\rightarrow w_{l+1}$, and the path is denoted by $\overleftarrow{P_{j}}^i$ if its edges are directed as $w_{l}\leftarrow w_{l+1}$.
Notice that the edges of $\overrightarrow{P}_{2^k-1}$ and $\overrightarrow{C_{2^k}}$ are oriented as defined in Section \ref{sec:preliminaries}.

From Sections \ref{sec:admissible_edges1} and \ref{sec:admissible_edges2} we know that
$\overrightarrow{\mathcal{G}}_{k}:=\overrightarrow{P}_{2^{k}-1} \cup H(\overrightarrow{\mathcal{G}}_{k})$  and
$\overleftarrow{\mathcal{G}}_{k}:=  \overrightarrow{P}_{2^{k}-1} \cup H(\overleftarrow{\mathcal{G}}_{k})$
with vertices $V$. We also know that $\overrightarrow{\Omega}_{2^{k}}:= \overrightarrow{C}_{2^{k}} \cup H_{\overrightarrow{C}}$
and $\overleftarrow{\Omega}_{2^{k}}:=  \overrightarrow{C}_{2^{k}} \cup H_{\overleftarrow{C}}$ with vertices $V \cup \{v_{2^k}\}$ where
$H_{\overrightarrow{C}}$ is as in (\ref{rightHC}) and $H_{\overleftarrow{C}}$ is as in (\ref{leftHC}) .

We say that a directed graph $ G$ is \emph{admissible} for a directed graph $\Gamma$ if $\chi_r\left( \Gamma \cup G \right) = \chi_r(\Gamma)$, and $G$
is \emph{forbidden} for $\Gamma$ if $\chi_r(\Gamma \cup G) > \chi_r(\Gamma)$.  As an example of these admissible  graphs and Lemma \ref{lemma_directcaterpillar},
see Figure \ref{figure3}.

\begin{figure} [htbp]
\begin{center}
\includegraphics[width=80mm]{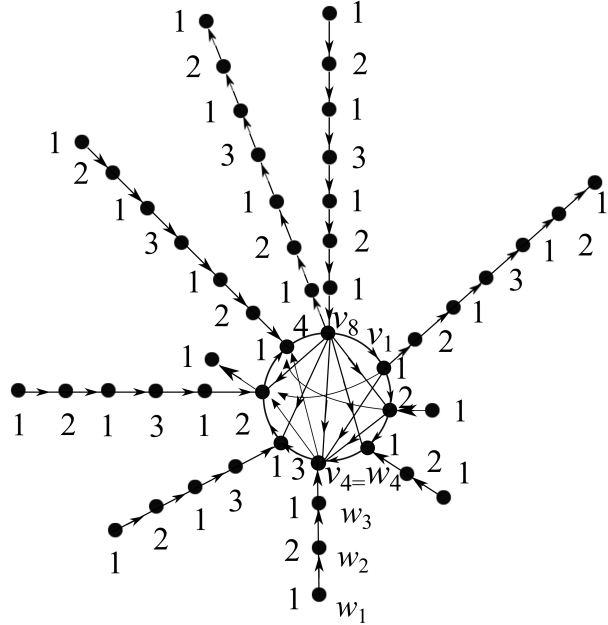}
\end{center}
\caption{A directed cycle with admissible directed paths.}
\label{figure3}
\end{figure}

\begin{lemma} \label{lemma_directcaterpillar} If $G_1$ is either $\overrightarrow{\mathcal{G}}_{k}$ or  $\overleftarrow{\mathcal{G}}_{k}$, and $G_2$ is either
$\ra{\Omega}_{2^{k}}$ or $\la{\Omega}_{2^{k}}$, then

\begin{enumerate}
\item the path $\overrightarrow{P_{j}}^i$ with $1<j\le i \le 2^k-1$ is an admissible  graph for $G_1$ and
\[\chi_r(G_1 \cup \overrightarrow{P_{j}}^i)= \chi_r(\overrightarrow{P}_{2^{k}-1})=k,\]

\item the path $\overleftarrow{P_{j}}^i$ with $1\le j \le  2^k-i$ is an admissible  graph for $G_1$ and
\[\chi_r(G_1 \cup \overleftarrow{P_{j}}^i)= \chi_r(\overrightarrow{P}_{2^{k}-1})=k,\]

\item the path $\overrightarrow{P_{j}}^i$ with $1<j\le i \le 2^k$ is an admissible  graph for $G_2$ and
\[\chi_r(G_2 \cup \overrightarrow{P_{j}}^i)= \chi_r(\overrightarrow{C}_{2^{k}})=k+1, \text{ and }\]

\item the path $\overleftarrow{P_{j}}^i$ with $1\le j \le  2^k-i$ is an admissible  graph for $G_2$ and
\[\chi_r(G_2 \cup \overleftarrow{P_{j}}^i)= \chi_r(\overrightarrow{C}_{2^{k}})=k+1.\]
\end{enumerate}
\end{lemma}

\begin{proof} We prove parts 1 and 2. Parts 3 and 4 are similar, and we omit the proofs. For both parts 1 and 2,  we prove the case
$G_1 = \overrightarrow{\mathcal{G}}_{k}$. Since the case  $G_1 = \overleftarrow{\mathcal{G}}_{k}$ is similar, it is omitted. We recall that
$V(\overrightarrow{\mathcal{G}}_{k})=\{v_{1} ,v_2, \ldots, v_{i}, v_{i+1}, \ldots , v_{2^k-1}\}$ is the set of vertices of  $\overrightarrow{\mathcal{G}}_{k}$ and that the
edges of $\overrightarrow{\mathcal{G}}_{k}$ are directed as  $v_{t}\rightarrow v_{t+1}$ for $t<2^{k}-1$.

{\bf Part 1.}  Suppose that $\{w_{1},\ldots,w_j=v_i\}$ is set of vertices of $\overrightarrow{P_{j}^i}$ for some fixed $1<j\le i \le 2^k-1$ where
$\overrightarrow{P_{j}^i}$  is the path attached to the vertex $v_i  \in G_1$ with edges directed as $w_{t}\rightarrow w_{t+1}$.

Corollary  \ref{uniqueranks} guarantees that $G_1$ has a unique minimal ranking. Let a minimal ranking function be
$f:V(G_1)\rightarrow\{1,2,\ldots,k\}$. Since
$1<j\le i \le 2^k-1$, we can label the vertices $\overrightarrow{P_{j}^i}$ with the labels given by the ranking function $f$ without
increasing the rank number of the new graph. Algorithm \ref{alg:label_right_paths} allows us to do it. (See for example Figure \ref{figure3}.)

\begin{algorithm}[htbp]
\begin{algorithmic}
    \For{$t = 0 \to j-1$}
        \State $f(w_{j-t})=f(v_{i-t});$
        \State $t \gets t + 1$
    \EndFor
\end{algorithmic}
\caption{Labeling vertices of $\protect\overrightarrow{P_{j}^i}$}
\label{alg:label_right_paths}
\end{algorithm}

Using Algorithm \ref{alg:label_right_paths} we obtain
\[ f(w_j)=f(v_i), \quad  f(w_{j-1})=f(v_{i-1}), \quad f(w_{j-2}) =f(v_{i-2}),\ldots,  \mbox{ and } \;\; f(w_1)=f(v_{i-j+1}).\]
 Therefore,  we have a ranking function $f:V(G_1 \cup \overrightarrow{P_{j}}^i)\rightarrow\{1,2,\ldots,k\}$ that labels all vertices of
 $G_1 \cup \overrightarrow{P_{j}}^i$ without increase the rank number.  This proves that
 $\chi_{r}(G_1 \cup \overrightarrow{P_{j}}^i) = \chi_r(\overrightarrow{P}_{2^{k}-1})=k$.

{\bf Part 2.}  Suppose that $\{w_{1},\ldots,w_j=v_i\}$ is the set of vertices of $\overleftarrow{P_{j}^i}$ for some fixed $1\le j \le 2^k-i$ where
$\overleftarrow{P_{j}^i}$  is the path attached to the vertex $v_i  \in G_1$ with edges directed as $w_{t}\leftarrow w_{t+1}$.

Corollary  \ref{uniqueranks} guarantees that $G_1$ has a unique minimal ranking.  Let a minimal ranking function be
$f:V(G_1)\rightarrow\{1,2,\ldots,k\}$.
Since  $1\le j\le  2^k-i$, we can label the vertices of $\overleftarrow{P_{j}^i}$ with the labels given by the ranking function $f$.
This labeling will not increase the rank number of the new graph. Algorithm \ref{alg:label_left_paths} allows us to do it. (See for example Figure \ref{figure3}.)

\begin{algorithm}[htbp]
\begin{algorithmic}
    \For{$t = 0 \to j-1$}
        \State $f(w_{j-t})=f(v_{i+t})$
        \State $t \gets t + 1$
    \EndFor
\end{algorithmic}
\caption{Labeling vertices of $\protect\overleftarrow{P_{j}^i}$}
\label{alg:label_left_paths}
\end{algorithm}

Using Algorithm \ref{alg:label_left_paths} we obtain
\[ f(w_j)=f(v_i), \quad f(w_{j-1})=f(v_{i+1}), \quad f(w_{j-2}) =f(v_{i+2}),\ldots, \mbox{ and }  \;\; f(w_1)=f(v_{i+j-1}).\]
Therefore,  we have a ranking function $f:V(G_1 \cup \overleftarrow{P_{j}}^i)\rightarrow\{1,2,\ldots,k\}$ that labels all vertices of
$G_1 \cup \overleftarrow{P_{j}}^i$ without increasing the rank number.  That is,
$\chi_{r}(G_1 \cup \overleftarrow{P_{j}}^i) = \chi_r(\overrightarrow{P}_{2^{k}-1})=k$.
\end{proof}

\begin{proposition} \label{directcaterpillar} If $G_1$ is either $\overrightarrow{\mathcal{G}}_{k}$ or  $\overleftarrow{\mathcal{G}}_{k}$,
and $G_2$ is either $\ra{\Omega}_{2^{k}}$ or $\la{\Omega}_{2^{k}}$, then

\begin{enumerate}
\item $\displaystyle{\chi_r\left(G_1\cup \bigcup_{1<j\le i \le 2^k-1} \overrightarrow{P_{j}}^i   \bigcup_{1\le j \le  2^k-i} \overleftarrow{P_{j}}^i\right)= \chi_r(\overrightarrow{P}_{2^{k}-1})=k}$, and

\item
$\displaystyle{\chi_r\Bigg(G_2 \cup \bigcup_{1<j\le i \le 2^k} \overrightarrow{P_{j}}^i  \bigcup_{{\begin{array}{c}
  1< j \le  2^k-i;   \\
  i <  2^k
\end{array}}}
\overleftarrow{P_{j}}^i  \Bigg)= \chi_r(\overrightarrow{C}_{2^{k}})=k+1}$.

\end{enumerate}
\end{proposition}

\begin{proof}  We prove part 1 for the case $G_1= \overrightarrow{\mathcal{G}}_{k}$. The other case where $G_1= \overleftarrow{\mathcal{G}}_{k}$
and part 2 are similar and thus omitted.  From Lemma \ref{lemma_directcaterpillar} part 1 we know that
$\chi_r( G_1)=\chi_r(\overrightarrow{P}_{2^{k}-1})=k$. Let $f:V(G_1)\rightarrow\{1,2,\ldots,k\}$ be the  minimal ranking of
$G_1=\overrightarrow{\mathcal{G}}_{k}$. Using Algorithms \ref{alg:label_right_paths} and \ref{alg:label_left_paths}  developed in the proof of Lemma \ref{lemma_directcaterpillar}, we define a
ranking function $f'$ that labels all vertices of all paths of the form of $\overrightarrow{P_{j}^i}$ or of the form $\overleftarrow{P_{j}^i}$
attached to $G_1$.  That is, $f'$ is the function defined by Algorithm \ref{alg:label_right_paths} if the path is of the form $\overrightarrow{P_{j}^i}$ and
$f'$ is the function defined by Algorithm \ref{alg:label_left_paths} if the path is of the form $\overleftarrow{P_{j}^i}$. From those algorithms is easy to see
that $f'(v) \le k$.

Let $f^*:V(D)\rightarrow\{1,2,\ldots,k\}$ be the function defined as
\[f^*(v) = \left\{
  \begin{array}{ll}
        f(v) & \hbox{if  \; $v\in G_1$ }\\
        f'(v),& \hbox{if \; $v\not \in G_1$,}
  \end{array}
\right.
\]
where
\[
D:=G_{1} \cup \bigcup_{1<j\le i \le 2^k-1} \overrightarrow{P_{j}}^i \cup  \bigcup_{1\le j \le  2^k-i} \overleftarrow{P_{j}}^i.
\]
From the definition of $f$ and $f'$ is easy to see that  $f^*(v) \le k$ for $v \in V(D)$.

We now prove that $f^*$ is a ranking function of $D$. That is, we want to prove that given any two vertices in $D$ with the same label, every directed path connecting those
two vertices has  a vertex with larger label.  We prove it by contradiction. Suppose that there are two vertices $u, w \in D$ connected by a directed path $P$
with  $f^*(w)=f^*(u)$ and for every other vertex $v' \in P$, we have $f^*(v') < f^*(u)$.

Let $V(P)= \{w=w_1, w_2, \ldots, w_j=v_i, v_{i+1}, \ldots, v_{l}=u_1, u_2, \ldots, u_r=u  \}$ be the set of vertices of $P$,
where $\{v_i, v_{i+1}, \ldots, v_{l}\}$ are vertices in $G_1$ and $V(P)\setminus  \{v_i, v_{i+1}, \ldots, v_{l}\}$ are vertices in $D \setminus G_1$.
Notice that $j$ and $r$ may be equal to one and that $i$ and $l$  may be equal. We suppose that the edges of $P$ are of the form
$ v_{t}\rightarrow v_{t+1}$, $ w_{t}\rightarrow w_{t+1}$, and $ u_{t}\rightarrow u_{t+1}$.

From  Algorithm \ref{alg:label_right_paths} in the proof of Lemma \ref{lemma_directcaterpillar},  we know that
$f'(w_{j-s})=f(v_{i-s})$ for $s=1, 2, \ldots , j-1$, and from Algorithm \ref{alg:label_left_paths} we know that $f'(u_{j+s})=f(v_{l+s})$ for $s=1, 2, \ldots r-1$.
These imply that there exists a path $P'$ with  vertices
$$V(P')= \{ v_{i-(j-1)}, \ldots, v_{i-1}, v_{i}, v_{i+1}, \ldots, v_{l}, v_{l+1},  \ldots, v_{l+r}\}$$
in $G_1$ satisfying that  $f(v_{i-(j-1)})=f^*(w) = f^*(u)=f(v_{l+r})$ and  $f(v')< f^*(w)$ for $v' \in V(P')\setminus \{v_{i-(j-1)},  v_{l+r}  \}$. That
is contradiction because $f$ is a ranking function of $G_1$ and
$f(v_{i-(j-1)})=f(v_{l+r})$.  This proves that $f^*$ is a ranking function for $D$. The definition of $f^*$ tells us that   $f^*(v) \le k$ for $v \in V(D)$.
Thus, $\chi_r( D)=\chi_r(\overrightarrow{P}_{2^{k}-1})=k$. This proves part 1 with $G_1= \overrightarrow{\mathcal{G}}_{k}$.
\end{proof}

Recall that $(\alpha_{r} \alpha_{r-1} \cdots \alpha_{1}\alpha_{0})_{2}$ with $\alpha_{h} = 0 \text{ or } 1$ for $0 \le h \le r$
is the binary  representation of a positive integer $b$ if $b=\alpha_{r}2^{r}+\alpha_{r-1}2^{r-1} + \cdots+\alpha_{1} 2^{1}+\alpha_{0}2^{0}$. We define $h$ as
$h(b)=j$ if $\alpha_{j}$ is the rightmost nonzero entry of the binary representation of $b$.  Fl\'{o}rez and Narayan \cite{floreznarayan} proved that if $v_m$ is a
vertex of $P_{2^k-1}$ in position  $m$, then $h(m)=f(v_m)$ where $f$ is the ranking function of $\overrightarrow{P}_{2^k-1}$. The same result extends naturally to
directed paths.

Let $W_1$ be a subgraph of a graph $W_2$. A \textit{vertex of attachment} of $W_1$ in $W_2$ is a vertex of $W_1$ that is incident with some edge of
$W_2$ that is not an edge of $W_1$ (for this definition see \cite{tutte} page 11 section I.4).

Let $G'$ be either $\overrightarrow{\mathcal{G}}_{k}$ or $\overleftarrow{\mathcal{G}}_{k}$ with vertices $V(G')=\{v_1, v_2, \ldots , v_{2^k-1}\}$ and let  $G(t)$
be either $\overrightarrow{\Omega}_{2^{h(t)}}$ or $\overleftarrow{\Omega}_{2^{h(t)}}$.  We define
\[\displaystyle{D(t):=G(t) \cup \bigcup_{1<j\le i \le 2^{h(t)}} \overrightarrow{P_{j}}^i \cup  \bigcup_{{\begin{array}{c}
  1< j \le  2^{h(t)}-i   \\
  i <  2^{h(t)}
\end{array}}}
\overleftarrow{P_{j}}^i }\]
for some $t\in \{2, \ldots, 2^k-1 \}$.

Let $$\Gamma:= \displaystyle{G' \cup \bigcup_{t=2}^{2k-1} D(t)}.$$ Notice that $D(t)$  has exactly one vertex of attachment in
$\Gamma$ which is given by  $v_t \in G'$. As an example of this graph and Proposition \ref{directcaterpillar2}, see Figure \ref{figure4}.

\begin{figure} [htbp]
\begin{center}
\includegraphics[width=80mm]{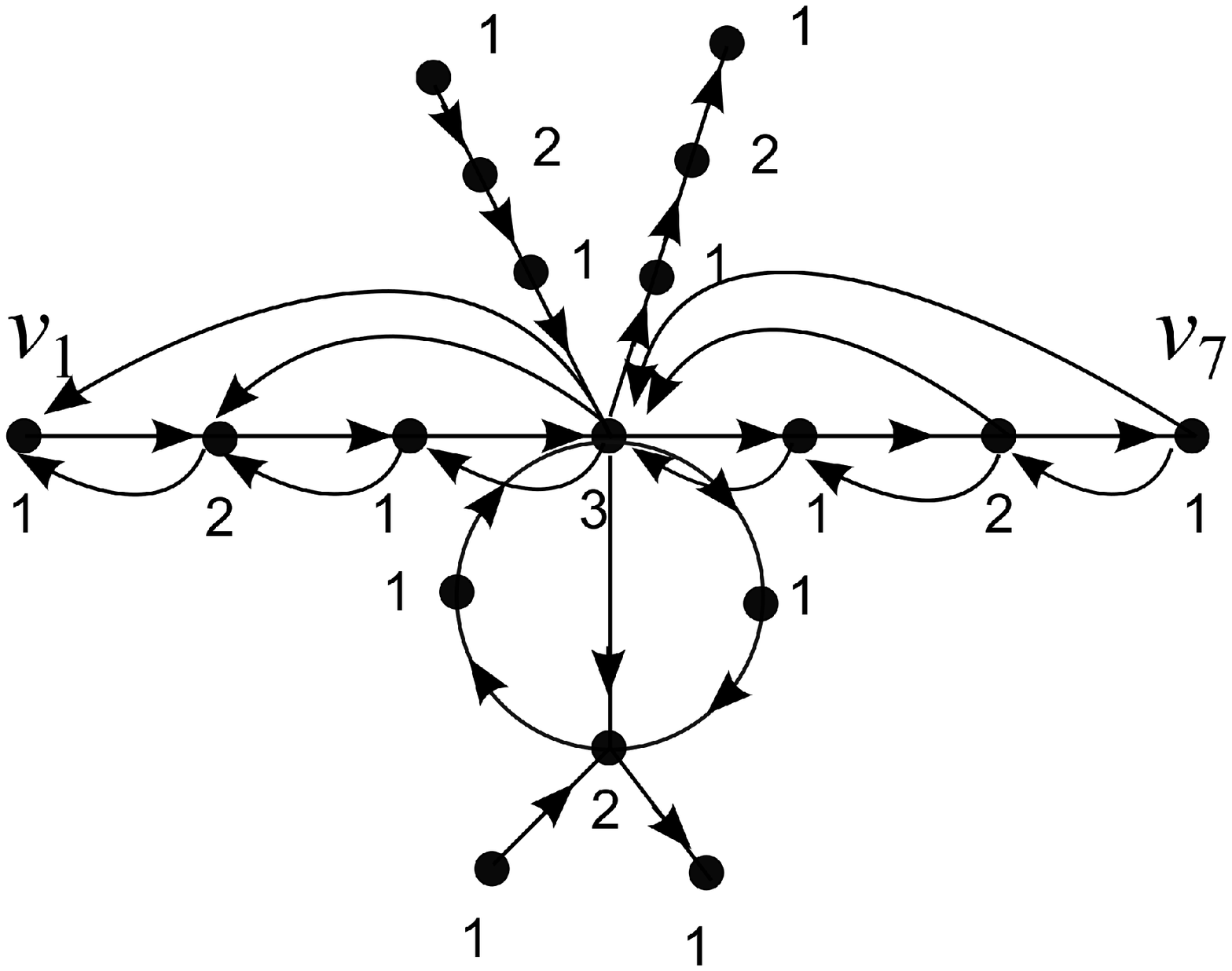}
\end{center}
\caption{A directed path with admissible directed graphs.} \label{figure4}
\end{figure}

Let $\Gamma'$ be the graph formed by $G'$ and the union of a set of graphs $N_t$ for $t\in I $, where $I$ is an index set, such that the graphs $G' $ and $N_t$
intersect exactly in the vertex $v_i \in  V(G')$. That is, $N_t$ has exactly one vertex of attachment $v_i \in \Gamma'$.   Theorem  \ref{general_directcaterpillar} proves that
$\Gamma'$ generalizes Lemma \ref{lemma_directcaterpillar} and Propositions  \ref{directcaterpillar} and \ref{directcaterpillar2}.

\begin{proposition} \label{directcaterpillar2}
If for each $t \in \{2,4, \ldots, 2^k-2  \}$ the set $\{v_t\} \subset \Gamma$ is the maximum set of vertices of attachment of
\[\displaystyle{G(t) \cup \bigcup_{1<j\le i \le 2^{h(t)}} \overrightarrow{P_{j}}^i \cup
 \bigcup_{{\begin{array}{c}
  				1< j \le  2^{h(t)}-i   \\
  				i <  2^{h(t)}
				\end{array}}}
\overleftarrow{P_{j}}^i }\]
in $\Gamma$, then
$\chi_r (\Gamma)= \chi_r(\overrightarrow{P}_{2^{k}-1})=k.$
\end{proposition}

\begin{proof} From Proposition \ref{directcaterpillar}, we know that $\chi_r( D(t))= \chi_r(\overrightarrow{C}_{2^{h(t)}}) \le h(t)$ where

\[\displaystyle{D(t):=G(t) \cup \bigcup_{1<j\le i \le 2^{h(t)}} \overrightarrow{P_{j}}^i \cup  \bigcup_{{\begin{array}{c}
  1< j \le  2^{h(t)}-i   \\
  i <  2^{h(t)}
\end{array}}}
\overleftarrow{P_{j}}^i}\]
for some $t \in \{2,4, \ldots, 2^k-2  \}.$ From the definition of $D(t)$ we can see that  $v_t = G' \cap D(t)$. Note that  $f(v_t )=h(t)$ and that  every other vertex of $D(t)$ has label
less than $h(t)$. So, attaching $D(t)$ to $G'$ does not increase the rank number of $G'$. Since this argument is true for every $t \in \{2,4, \ldots, 2^k-2  \}$,  it proves that
$\chi_r( \Gamma)= \chi_r(\overrightarrow{P}_{2^{k}-1})=k$.
\end{proof}

\begin{theorem} \label{general_directcaterpillar}  Let $v_i$ be the vertex of attachment  of $N_t$ in $\Gamma'$ for $t$ in an index set $I$ and some $i \in \{2, \ldots, 2^k-2\}$. Suppose that
$\chi_r(N_t)= h(i)$. If there is a ranking function $f_t$ of $N_t$ such that $f_t(v_i)=h(i)$, then $\chi_r(\Gamma')=\chi_r(\overrightarrow{P}_{2^{k}-1} ).$
\end{theorem}

\begin{proof} Since the vertex $v_i = G' \cap N_t$ has label $h(i)=f_t(v_i)=f(v_i)$ where $f_t$ and $f$ are the ranking functions of $N_t$ and $G'$, respectively, every other vertex of $N_t$
has label less than $h(i)$. So, attaching $N_t$ to $G'$ does not increase the rank number of $\Gamma'$. Since this argument is true for every
$t \in \{2,4, \ldots, 2^k-2  \}$,  it proves that  $\chi_r( \Gamma')= \chi_r(\overrightarrow{P}_{2^{k-1}})$.
\end{proof}

\end{document}